\documentclass[12pt]{article}
\usepackage{enumitem, amsmath, amsthm, amssymb, yhmath, geometry, cleveref, biblatex, asymptote}
\DeclareMathOperator*{\argmax}{\mathrm{arg\,max}}
\newcommand{\floor}[1]{\lfloor #1 \rfloor}
\newcommand{\rsg}{\rightsquigarrow}

\theoremstyle{definition}
\newtheorem{theorem}{Theorem}[section]
\newtheorem{lemma}[theorem]{Lemma}
\newtheorem{definition}[theorem]{Definition}
\newtheorem{claim}[theorem]{Claim}

\title{Competitively Constructed Planar Graphs}
\author{Wesley Pegden\thanks{Department of Mathematical Sciences, Carnegie Mellon University. Email: \texttt{wes@math.cmu.edu}} \and Eric Wang\thanks{School of Computer Science, Carnegie Mellon University. Email: \texttt{ejwang2@andrew.cmu.edu}}}
\date{October 2025}

\begin{document}

\maketitle

\begin{abstract}
We introduce and study two Maker-Breaker-like games for constructing planar graphs: the \textit{edge drawing game}, where two players take turns drawing non-intersecting edges between points in the plane, and the \textit{circle packing game}, where the players take turns placing disjoint circles in the plane. Both games produce planar graphs: the edge drawing game results in a plane graph drawing, and the circle packing game yields a planar graph via the contact graph of the packing. For both games, we give necessary conditions under which a given planar graph can be constructed. We also show that the two games are indeed different by giving a class of graphs which can be constructed in one but not the other. 
\end{abstract}

\section{Introduction}
A Maker-Breaker game is a two player game played on a hypergraph $\mathcal H$ where the players, Maker and Breaker, take turns occupying points of $\mathcal H$. Maker's goal is to fully occupy an edge; Breaker's goal is to prevent this. 

There is a surprising connection, first noted by Erdős and Selfridge \cite{ES73}, between optimal play in Maker-Breaker games and randomness. A striking example due to Beck \cite{Bec08} is the Maker-Breaker game played on the edge set of the complete graph $K_n$. Maker wins if he occupies the edges of a $K_q$. Evidently, Maker has a winning strategy for small enough $q$. What is the threshold $q(n)$ such that Maker has a strategy to occupy a $K_{q(n)}$ but not a $K_{q(n)+1}$? It turns out that for sufficiently large $n$, this threshold $q(n)$ coincides, up to an additive constant, with the clique number of the random graph $G_{n, 1/2}$. 

Motivated by this connection, we study competitively constructed planar graphs. We consider two ways of constructing planar graphs. In the first, the players take turns drawing non-intersecting curves between pairs of points in the plane; the second involves taking turns drawing internally disjoint circles in the plane, thus constructing a circle packing, which has an associated contact graph. 

We now discuss the setup of these games more precisely. The first game we consider is the edge drawing game. Fix $n$ points in the plane. Two players, whom we call Builder and Spoiler, take turns drawing a non-self-intersecting curve between two distinct points. Distinct pairs of points must be chosen on each turn, and no two curves may intersect except at endpoints. In this way, Builder and Spoiler together construct a drawing of a planar graph. The game ends when no more edges can be drawn, i.e. after exactly $3n-6$ moves. 

Sometimes we will consider a variant of this procedure where Builder (resp. Spoiler) may draw $\beta$ edges in one turn; in such cases we say that Builder (resp. Spoiler) is given a $\beta : 1$ (resp. $1 : \beta$) bias. 

The first game we analyze is the \textit{Hamiltonian cycle game}. Let $G$ be the planar graph obtained at the end of the game. Builder wins the Hamiltonian cycle game if $G$ is Hamiltonian, and otherwise Spoiler wins. We show that the winner of this game depends on the bias. 
\begin{theorem}\label{thm:ham21}
In a $2:1$ biased game, Builder wins the Hamiltonian cycle game. 
\end{theorem}

On the other hand:
\begin{theorem}\label{thm:ham13}
In a $1:3$ biased game, Spoiler wins the Hamiltonian cycle game for $n$ sufficiently large.
\end{theorem}

Now, we ask if Builder can force the existence of a vertex of high degree in the subgraph induced by the edges Builder draws. Given adjacent vertices $u$ and $v$, say that $u$ and $v$ are \textit{Builder-adjacent} if Builder drew the edge $\{u, v\}$ at some point during the game. Analogously, define the \textit{Builder-degree} of a vertex $u$ to be the number of vertices $v$ such that $u$ and $v$ are Builder-adjacent. We show that Builder can indeed force a constant fraction of the total number of vertices to be Builder-adjacent to a pre-selected vertex. Moreover, the strategy Builder adopts to achieve this turns out to guarantee that the resulting graph has constant diameter. 
\begin{theorem}\label{thm:degree}
Builder has a strategy wherein:
\begin{itemize}
    \item He can nominate a vertex $v$ before the start of the game and ensure that $v$ has Builder-degree $cn$ by the end of the game, for an absolute constant $c$. 
    \item The graph at the end of the game has diameter $d$ for an absolute constant $d$. 
\end{itemize}

\end{theorem}

Next, we are interested in whether, given an arbitrary planar graph $H$, Builder has a strategy such that $H \subseteq G_i$ for some $i$, where $G_i$ is the graph obtained after $i$ moves. We consider a simplification where there the game is instead played on countably infinite isolated points in the plane---call this the \textit{$H$-subgraph edge drawing game}. \Cref{thm:degree} shows that Builder wins the $K_{1,r}$-subgraph edge drawing game for all $r$. Can Builder construct any planar graph? 

The answer is no, as our next result shows. An \textit{Apollonian network} is a planar 3-tree. That is, they are defined recursively as follows:
\begin{itemize}
    \item $K_3$ is an Apollonian network.
    \item Let $G$ be an Apollonian network. Choose a face $F$ of $G$, add a vertex $v$ in $F$, and add edges from $v$ to the three vertices incident to $F$. The resulting graph $G'$ is an Apollonian network. 
\end{itemize}
\begin{theorem}\label{thm:apollonian}
Let $H$ be a connected planar graph, and suppose Builder wins the $H$-subgraph edge drawing game. Then $H$ is a subgraph of an Apollonian network. 
\end{theorem}
It is known \cite{APC90} that the set of \textit{partial 3-trees}
\[
\{H \mid \text{$H \subseteq G$ for some Apollonian network $G$}\}
\]
is characterized by the following set of four forbidden minors: $K_5$, $K_{3,3}$, the octahedron, and the pentagonal prism. It follows from \Cref{thm:apollonian} that Builder cannot hope to construct the octahedron or the pentagonal prism (of course, $K_5$ and $K_{3,3}$ are not planar and a fortiori cannot be constructed by Builder).

The second game we consider is the circle packing game. A \textit{circle packing} is a set of circles $\mathcal C = \{\Omega_1, \ldots, \Omega_n\}$ in the plane whose interiors are disjoint. The \textit{contact graph} of $\mathcal C$, denoted $\mathbf G(\mathcal C)$, is the graph on $V(\mathbf G(\mathcal C)) = \{1, \ldots, n\}$ and $\{i, j\} \in E(\mathbf G(\mathcal C))$ if and only if $\Omega_i$ is tangent to $\Omega_j$. An example is shown in \Cref{fig:contactgraph}.

We will abuse notation and also use $\mathbf G(\mathcal C)$ to denote the planar graph drawing induced by $\mathcal C$. That is, place a vertex at the center of each circle in $\mathcal C$, and draw a line segment between two vertices $u$ and $v$ if the corresponding circles centered at $u$ and $v$ are tangent. 

The Koebe-Andreev-Thurston circle packing theorem \cite{PA95} states that for any planar graph $H$, there exists a circle packing $\mathcal C$ such that $\mathbf G(\mathcal C) \cong H$. Motivated by this result, we ask which planar graphs can be competitively constructed in the context of circle packings. 

\begin{figure}
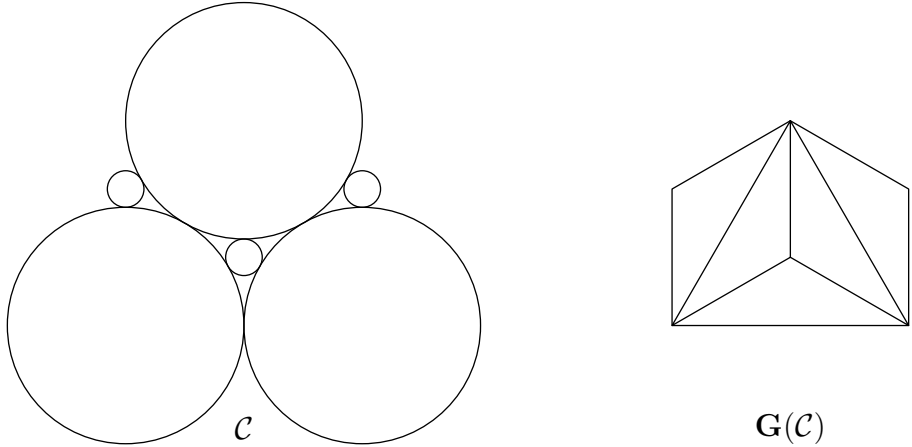

\begin{asy}
size(12cm);
picture left = new picture; picture right = new picture;
import geometry;

pair A = dir(90), B = dir(210), C = dir(330), O = (0, 0), D = (B+C)/2, E = (C+A)/2, F = (A+B)/2, P = reflect(A, B) * O, Q = reflect(A, C) * O;
real r = abs(A - B)/2; 
guide omega = circle(O, 1 - r);
draw(left, circle(A, r)^^circle(B, r)^^circle(C, r)^^omega^^circle(P, 1 - r)^^circle(Q, 1 - r));
label(left, "$\mathcal C$", 2.5 * D);

draw(right, A--P--B--C--Q--A^^A--B--O--C--A--O);
label(right, "$\mathbf G(\mathcal C)$", 2.5 * D);

add(shift(0,0)*left); add(shift(4,0)*right);
\end{asy}
\centering
\caption{Contact graph of a circle packing}
\label{fig:contactgraph}
\end{figure}

More precisely, Builder and Spoiler take turns placing circles in the plane with disjoint interiors. Let $\mathcal C_i$ denote the circle packing after Builder and Spoiler have each played $i$ moves. For a planar graph $H$, the \textit{$H$-subgraph circle packing game} is as follows: Builder wins if $H \subseteq \mathbf G(\mathcal C_i)$ for some $i$, and Spoiler wins otherwise. 
\begin{theorem}\label{thm:circle}
Let $H$ be a planar triangulation on at least 3 vertices. If Builder can win the $H$-subgraph circle packing game, then $H$ is an Apollonian network. 
\end{theorem}
We also consider a biased version of the game. Let $\varepsilon \in (0, 1)$. In the $(1+\varepsilon):1$ \textit{biased $H$-subgraph circle packing game}, Builder and Spoiler still take turns placing circles, but every $\floor{1/\varepsilon}$ moves, Builder is given one additional move. The goal of each player remains the same. We show that any bias whatsoever gives Builder unlimited power:
\begin{theorem}\label{thm:biasedcircle}
Let $\varepsilon \in (0, 1)$. For any planar graph $H$, Builder wins the $(1+\varepsilon) : 1$ biased $H$-subgraph circle packing game. 
\end{theorem}

Finally, a natural question to ask is whether the edge drawing game and the circle packing game are different. Are there graphs $H$ for which Builder wins one but not the other? Our next result shows that, in the circle packing game, Builder has a strategy to construct arbitrarily large Apollonian networks. 

\begin{theorem}\label{thm:diameter}
Let $n \geq 3$. In the unbiased circle packing game, Builder has a strategy to construct an Apollonian network on $n$ vertices. 
\end{theorem}

It is known that the diameter of an Apollonian network grows linearly with the maximum depth of a face \cite{FT14}, so \Cref{thm:diameter} implies that Builder can construct graphs of arbitrarily large diameter. However, \Cref{thm:degree} shows that Builder cannot construct graphs of arbitrarily large diameter in the edge drawing game: if Spoiler adopts Builder's strategy in \Cref{thm:degree}, then Spoiler can force the final graph to have constant diameter. 

\section{Proof of \Cref{thm:ham21}}
There is a unique planar triangulation on 5 vertices, and this graph is Hamiltonian, so we may assume that $n \geq 6$. Builder employs the following strategy. 
\begin{itemize}
\item If no edges have been drawn (i.e. it is the first move of the game), then Builder draws a $K_{1,2}$.
\item If Spoiler draws an edge $\{u, v\}$ where the vertices $u$ and $v$ are isolated, and there exists an isolated vertex $x$, then Builder chooses an isolated vertex $x$ and draws edges $\{u,x\}$ and $\{v,x\}$, making sure not to surround any vertices while doing so.  
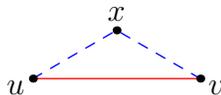
\begin{center}
\begin{asy}
size(3cm);
pair A = (0,0), B = dir(210), C = dir(-30);
draw(B--C, red); draw(B--A--C, blue+dashed);
dot("$x$", A, dir(90)); dot("$u$", B, dir(B)); dot("$v$", C, dir(C));
\end{asy}
\end{center}
\item If Spoiler joins an isolated vertex $x$ to a 3-cycle $(u, v, w)$, say via the edge $\{x, w\}$, Builder draws the edges $\{x, u\}$ and $\{x, v\}$ in such a way that there are no vertices inside the newly formed faces bounded by $(x, w, u)$ and $(x, w, v)$. 
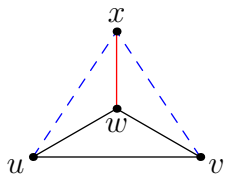
\begin{center}
\begin{asy}
size(3cm);
pair A = dir(90), B = dir(150), C = dir(30), D = 1.8*A;
draw(A--B--C--A); draw(A--D, red); draw(B--D--C, blue+dashed);
dot("$w$", A, dir(-90)); dot("$u$", B, dir(210)); dot("$v$", C, dir(-30)); dot("$x$", D, dir(90));
\end{asy}
\end{center}
\item If Spoiler joins an isolated vertex $x$ to a $K_{1,2}$ with vertices $u,v,w$ and $u$ not adjacent to $v$, which can be done in two distinct ways, then Builder responds as pictured below in such a way that the newly formed faces bounded by $(u, v, w)$ and $(x, v, w)$ do not contain any vertices other than $u$. 
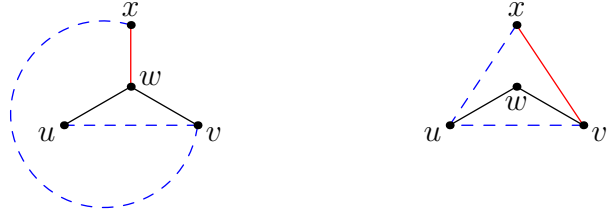
\begin{center}
\begin{asy}
size(8cm);
pair A = dir(90), B = dir(150), C = dir(30), D = 1.8*A, T = (5, 0), A1 = A+T, B1 = B+T, C1 = C+T, D1 = D+T;
draw(B--A--C); draw(D--A, red); draw(B--C, blue+dashed); draw(D..B+(-0.5, -0.5)..C, blue+dashed);
dot("$w$", A, dir(30)); dot("$u$", B, dir(210)); dot("$v$", C, dir(-30)); dot("$x$", D, dir(90));
draw(B1--A1--C1); draw(D1--C1, red); draw(D1--B1--C1, blue+dashed);
dot("$w$", A1, dir(-90)); dot("$u$", B1, dir(210)); dot("$v$", C1, dir(-30)); dot("$x$", D1, dir(90));
\end{asy}
\end{center}
\item If Spoiler joins a 3-cycle $(u, v, w)$ to a $K_{1,2}$ with vertices $u_1,v_1,w_1$ and $u_1$ not adjacent to $v_1$, which can be done in two distinct ways, then Builder responds as pictured below, again in such a way that the enclosed regions do not contain any other vertices. 
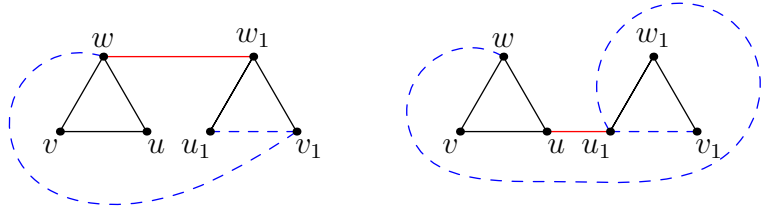
\begin{center}
\begin{asy}
picture l = new picture; picture r = new picture;
size(10cm);
pair A = dir(90), B = dir(210), C = dir(-30), T = (3, 0), A1 = A+T, B1 = B+T, C1 = C+T;
draw(l, A..B+(-1, 0)..B1+(0, -1)..C1, blue+dashed);
draw(l, A--B--C--A^^A1--B1--A1--C1); draw(l, A--A1, red); draw(l, B1--C1, blue+dashed);
draw(r, A..B+(-1, 0)..C+(0,-1)..C1+(1,0)..A1+(1,1)..B1, blue+dashed);
draw(r, A--B--C--A^^A1--B1--A1--C1); draw(r, C--B1, red); draw(r, B1--C1, blue+dashed);
dot(l, "$w$", A, dir(90)); dot(l, "$v$", B, dir(240)); dot(l, "$u$", C, dir(-60));
dot(l, "$w_1$", A1, dir(90)); dot(l, "$u_1$", B1, dir(240)); dot(l, "$v_1$", C1, dir(-60));
dot(r, "$w$", A, dir(90)); dot(r, "$v$", B, dir(240)); dot(r, "$u$", C, dir(-60));
dot(r, "$w_1$", A1, dir(90)); dot(r, "$u_1$", B1, dir(240)); dot(r, "$v_1$", C1, dir(-60));
add(shift(0,0)*l); add(shift(8,0)*r);
\end{asy}
\end{center}
\item If Spoiler joins two 3-cycles $(u, v, w)$ and $(u_1, v_1, w_1)$, say via the edge $\{u_1, u_2\}$, Builder draws the edges $\{w, w_1\}$ and $\{w, v_1\}$, as usual ensuring that there are no vertices in the faces bounded by $(u_1, u, v, w, v_1)$ and $(w, u, u_1, w_1)$. 
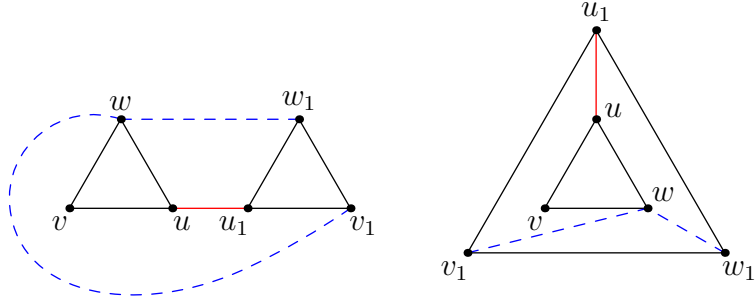
\begin{center}
\begin{asy}
picture l = new picture; picture r = new picture;
size(10cm);
pair A = dir(90), B = dir(210), C = dir(-30), T = (3, 0), A1 = A+T, B1 = B+T, C1 = C+T, A2 = 2.5*A, B2 = 2.5*B, C2 = 2.5*C;
draw(l, A..B+(-1, 0)..B1+(0, -1)..C1, blue+dashed);
draw(l, A--B--C--A^^A1--B1--C1--A1); draw(l, C--B1, red); draw(l, A--A1, blue+dashed); 
dot(l, "$w$", A, dir(90)); dot(l, "$v$", B, dir(240)); dot(l, "$u$", C, dir(-60));
dot(l, "$w_1$", A1, dir(90)); dot(l, "$u_1$", B1, dir(240)); dot(l, "$v_1$", C1, dir(-60));
draw(r, A--B--C--A^^A2--B2--C2--A2); draw(r, A--A2, red); draw(r, B2--C--C2, blue+dashed);
dot(r, "$u$", A, dir(30)); dot(r, "$v$", B, dir(240)); dot(r, "$w$", C, dir(50));
dot(r, "$u_1$", A2, dir(90)); dot(r, "$v_1$", B2, dir(240)); dot(r, "$w_1$", C2, dir(-60));
add(shift(0,0)*l); add(shift(8,0)*r);
\end{asy}
\end{center}
\end{itemize}
Now suppose we fall into none of the above cases. It's easy to see that each connected component $H$ of the current graph $G$ is either an isolated vertex, a $K_{1,2}$, or is such that each face of $H$ which contains other components of $G$ is a 3-cycle; moreover, there is at least one of the latter kind of component since $n \geq 6$---this situation may occur, for instance, on the very first move when Builder draws a $K_{1,2}$ and Spoiler completes the 3-cycle, enclosing other points while doing so. Builder should then proceed as follows. 
\begin{itemize}
\item If $G$ is connected, then Builder can play arbitrarily until the end of the game. 
\item If there is an isolated vertex $x$, join it to a 3-cycle $(u, v, w)$ by drawing edges $\{u, x\}$ and $\{x, v\}$ such that the face bounded by $(u, x, v, w)$ contains no vertices. 
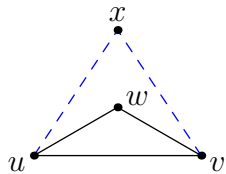
\begin{center}
\begin{asy}
size(3cm);
pair A = dir(90), B = dir(150), C = dir(30), D = 1.8*A;
draw(A--B--C--A); draw(B--D--C, blue+dashed);
dot("$w$", A, dir(30)); dot("$u$", B, dir(210)); dot("$v$", C, dir(-30)); dot("$x$", D, dir(90));
\end{asy}
\end{center}
\item Else, if there are two 3-cycles $(u, v, w)$ and $(u_1, v_1, w_1)$ in different components, Builder should merge them into one component by drawing the edges $\{w, w_1\}$ and $\{w, v_1\}$ such that the newly formed face contains no vertices other than $u$, $v$, and $u_1$. 
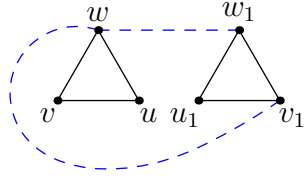
\begin{center}
\begin{asy}
size(4cm);
pair A = dir(90), B = dir(210), C = dir(-30), T = (3, 0), A1 = A+T, B1 = B+T, C1 = C+T;
draw(A..B+(-1, 0)..B1+(0, -1)..C1, blue+dashed);
draw(A--B--C--A^^A1--B1--C1--A1); draw(A--A1, blue+dashed); 
dot("$w$", A, dir(90)); dot("$v$", B, dir(240)); dot("$u$", C, dir(-60));
dot("$w_1$", A1, dir(90)); dot("$u_1$", B1, dir(240)); dot("$v_1$", C1, dir(-60));
\end{asy}
\end{center}
\item Otherwise, there must be exactly two connected components. One is a $K_{1,2}$, and the other is a graph whose face containing the $K_{1,2}$ is bounded by a 3-cycle. Builder joins these two components as pictured. 
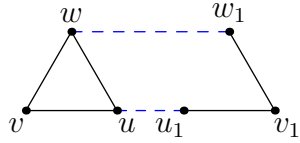
\begin{center}
\begin{asy}
size(4cm);
pair A = dir(90), B = dir(210), C = dir(-30), T = (3, 0), A1 = A+T, B1 = B+T, C1 = C+T;
draw(A--B--C--A^^A1--C1--B1); draw(A--A1^^C--B1, blue+dashed); 
dot("$w$", A, dir(90)); dot("$v$", B, dir(240)); dot("$u$", C, dir(-60));
dot("$w_1$", A1, dir(90)); dot("$u_1$", B1, dir(240)); dot("$v_1$", C1, dir(-60));
\end{asy}
\end{center}
\end{itemize}
We now show that this strategy works. 
\begin{definition}
Let $G$ be a plane drawing, and let $P$ be a graph property. We say that $G$ is \textit{eventually} $P$ if all plane triangulations $H$ with $V(H) = V(G)$ and $E(G) \subseteq E(H)$ satisfy $P$. 
\end{definition}
\begin{definition}
Let $G$ be a plane drawing whose outer face is a 3-cycle $(u_1, u_2, u_3)$. We say that $G$ is \textit{strongly Hamiltonian} if the following hold. 
\begin{itemize}
    \item $G$ has a Hamiltonian path from $u_i$ to $u_j$ for all $\{i, j\} \subseteq \{1, 2, 3\}$, $i \neq j$. In this case we use $u_i \stackrel G \rsg u_j$ to denote these paths. 
    \item $G - \{u_k\}$ has a Hamiltonian path from $u_i$ to $u_j$ for all $\{i, j, k\} = \{1, 2, 3\}$. In this case we use $u_i \stackrel G\rsg_{u_k} u_j$ to denote these paths, which we will call \textit{near-Hamiltonian} paths. 
\end{itemize}
\end{definition}
Note that $K_3$ is strongly Hamiltonian. We claim that throughout the game, Builder is able to maintain that all connected components whose outer face is a 3-cycle is eventually strongly Hamiltonian. It suffices to show inductively that this invariant holds in the following five configurations. 

\begin{enumerate}[label = \textit{Case \arabic*.}]

\item Let $(u, v, w)$ be the outer face of a component $H$ before $x$ becomes connected to it. By induction, assume that $H$ is eventually strongly Hamiltonian. Let $H'$ be the new component, i.e. $V(H') = V(H) \cup \{x\}$ and $E(H') = E(H) \cup \{\{u, x\}, \{w, x\}, \{v, w\}\}$. The outer face of $H'$ is $(u, v, x)$. Note that $H'$ eventually contains the edge $\{x, w\}$.
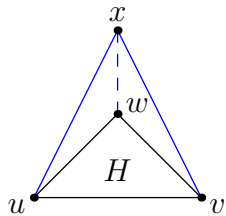
\begin{center}
\begin{asy}
size(3cm);
pair A = dir(90), B = dir(180), C = dir(0), D = 2*A;
draw(A--B--C--A); draw(B--D--C, blue); draw(D--A, blue+dashed);
dot("$w$", A, dir(30)); dot("$u$", B, dir(210)); dot("$v$", C, dir(-30)); dot("$x$", D, dir(90));
label("$H$", (A+B+C)/3);
\end{asy}
\end{center}
Then $H'$ is eventually strongly Hamiltonian, as witnessed by the following paths.
\begin{align*}
\text{Hamiltonian paths:} \quad &x \to u \stackrel H \rsg_v w \to v \\
&x \to v \stackrel H \rsg_u w \to u \\
&u \stackrel H \rsg_v w \to x \to v \\
\text{Near-Hamiltonian paths:} \quad &x \to w \stackrel H \rsg_u v \\
&x \to w \stackrel H \rsg_v u \\
&u \stackrel H \rsg v
\end{align*}

\item Let $(u, v, w)$ be the outer face of a component $H$, and let $(u_1, v_1, w_1)$ be a 3-cycle which contains $H$. Let $H_1$ be the graph whose outer face is the 3-cycle $(w, v_1, w_1)$. Finally, let $H'$ be the new component defined by $V(H') = V(H) \cup V(H_1) \cup \{u_1\}$ and $E(H') = E(H) \cup E(H_1) \cup \{\{u_1, v_1\}, \{u_1, u\}, \{u_1, w_1\}\}$. 
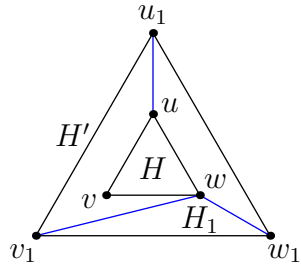
\begin{center}
\begin{asy}
size(4cm);
pair A = dir(90), B = dir(210), C = dir(-30), A2 = 2.5*A, B2 = 2.5*B, C2 = 2.5*C;
draw(A--B--C--A^^A2--B2--C2--A2); draw(A--A2^^B2--C--C2, blue);
dot("$u$", A, dir(30)); dot("$v$", B, dir(200)); dot("$w$", C, dir(50));
dot("$u_1$", A2, dir(90)); dot("$v_1$", B2, dir(240)); dot("$w_1$", C2, dir(-60));
label("$H$", (A+B+C)/3); label("$H'$", (A2+B2)/2 + (-0.4, 0)); label("$H_1$", C-(0,0.45));
\end{asy}
\end{center}
Assume inductively that $H$ and $H_1$ are eventually strongly Hamiltonian. Then $H'$ is as well, as witnessed by the following paths.
\begin{align*}
\text{Hamiltonian paths:} \quad &u_1 \to u \stackrel H \rsg w \stackrel{H_1}{\rsg} v_1 \\
\quad &u_1 \to u \stackrel H \rsg w \stackrel{H_1}{\rsg} w_1 \\
\quad &v_1 \to u_1 \to u \stackrel H \rsg w \stackrel{H_1}{\rsg}_{v_1} w_1 \\
\text{Near-Hamiltonian paths:} \quad &u_1 \to u \stackrel H \rsg w \stackrel{H_1}{\rsg}_{w_1} v_1 \\
&u_1 \to u \stackrel H \rsg w \stackrel{H_1}{\rsg}_{v_1} w_1 \\
&v_1 \stackrel{H_1}{\rsg}_{w_1} w \stackrel H \rsg u \to u_1 \to w_1
\end{align*}

\item Let $H$ and $H_1$ have outer faces $(u, v, w)$ and $(u_1, v_1, w_1)$ respectively which are positioned as shown. 
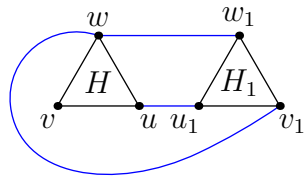
\begin{center}
\begin{asy}
size(4cm);
pair A = dir(90), B = dir(210), C = dir(-30), T = (3, 0), A1 = A+T, B1 = B+T, C1 = C+T;
draw(A..B+(-1, 0)..B1+(0, -1)..C1, blue);
draw(A--B--C--A^^A1--B1--C1--A1); draw(C--B1^^A--A1, blue); 
dot("$w$", A, dir(90)); dot("$v$", B, dir(240)); dot("$u$", C, dir(-60));
dot("$w_1$", A1, dir(90)); dot("$u_1$", B1, dir(240)); dot("$v_1$", C1, dir(-60));
label("$H$", (A+B+C)/3); label("$H_1$", (A1+B1+C1)/3);
\end{asy}
\end{center}
Let $H'$ be the merged component defined by $V(H') = V(H) \cup V(H_1)$ and $E(H') = E(H) \cup E(H_1) \cup \{\{u, u_1\}, \{w, w_1\}, \{w, v_1\}\}$. Assume inductively that $H$ and $H_1$ are eventually strongly Hamiltonian. Then $H'$ has the following Hamiltonian paths:
\begin{align*}
&w \stackrel H \rsg u \to u_1 \stackrel{H_1}\rsg w_1 \\
&w \stackrel H \rsg u \to u_1 \stackrel{H_1}\rsg v_1 \\
&w_1 \to w \stackrel H \rsg u \to u_1 \stackrel{H_1}\rsg_{w_1} v_1,
\end{align*}
and $H'$ has the following near-Hamiltonian paths:
\begin{align*}
&w \stackrel H \rsg u \to u_1 \stackrel{H_1}\rsg_{v_1} w_1 \\
&w \stackrel H \rsg u \to u_1 \stackrel{H_1}\rsg_{w_1} v_1.
\end{align*}
As $H'$ is drawn as above, it does not yet have a near-Hamiltonian path between $w_1$ and $v_1$. If the edge $\{v, v_1\}$ exists, then we would have the path
\[
w_1 \stackrel{H_1}\rsg_{v_1} u_1 \to u \stackrel H\rsg_w v \to v_1.
\]
Suppose $\{v, v_1\}$ does not exist. The only way this happens is that $v$ is surrounded by the edge $\{w, u_1\}$. 
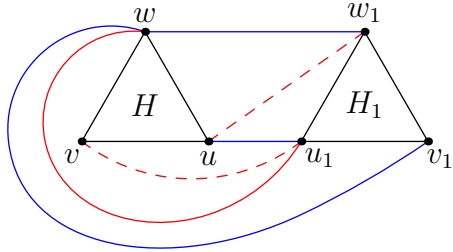
\begin{center}
\begin{asy}
size(6cm);
pair A = dir(90), B = dir(210), C = dir(-30), T = (3, 0), A1 = A+T, B1 = B+T, C1 = C+T;
draw(A..B+(-1, 0)..B1+(0, -1)..C1, blue);
draw(A--B--C--A^^A1--B1--C1--A1); draw(C--B1^^A--A1, blue); 
draw(A..B+(-0.5, 0)..C+(0, -1)..B1, red);
draw(B..C+(0, -0.5)..B1, red+dashed); draw(C--A1, red+dashed);
dot("$w$", A, dir(90)); dot("$v$", B, dir(240)); dot("$u$", C, dir(-90));
dot("$w_1$", A1, dir(90)); dot("$u_1$", B1, dir(-50)); dot("$v_1$", C1, dir(-60));
label("$H$", (A+B+C)/3); label("$H_1$", (A1+B1+C1)/3);
\end{asy}
\end{center}
But in this case $\{v, u_1\}$ and $\{u, w_1\}$ are inevitable, giving rise to the path
\[
w_1 \to u \stackrel H\rsg_w v \to u_1 \stackrel{H_1}\rsg_{w_1} v_1.
\]
Thus $H'$ is eventually strongly Hamiltonian. 

\item Assume the same setup as in case 3 but without the edge $\{u, u_1\}$. If $\{u, u_1\}$ were to eventually exist, then we are done by case 3. The only way to prevent this is to surround $u$ by drawing $\{v, w_1\}$. Then $\{u, w_1\}$ and $\{v, u_1\}$ will eventually be drawn.
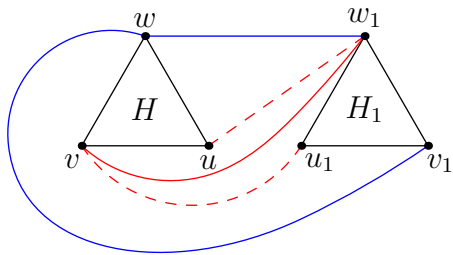
\begin{center}
\begin{asy}
size(6cm);
pair A = dir(90), B = dir(210), C = dir(-30), T = (3, 0), A1 = A+T, B1 = B+T, C1 = C+T;
draw(A..B+(-1, 0)..B1+(0, -1)..C1, blue);
draw(A--B--C--A^^A1--B1--C1--A1); draw(A--A1, blue); 
draw(B..C+(0, -0.4)..B1+(0, 0.5)..A1, red);
draw(C--A1, red+dashed);
draw(B..C+(0, -0.8)..B1, red+dashed);
dot("$w$", A, dir(90)); dot("$v$", B, dir(240)); dot("$u$", C, dir(-90));
dot("$w_1$", A1, dir(90)); dot("$u_1$", B1, dir(-50)); dot("$v_1$", C1, dir(-60));
label("$H$", (A+B+C)/3); label("$H_1$", (A1+B1+C1)/3);
\end{asy}
\end{center}
Thus $H'$ is eventually strongly Hamiltonian, as witnessed by the following paths.
\begin{align*}
\text{Hamiltonian paths:} \quad &w \stackrel H\rsg v \to u_1 \stackrel{H_1}\rsg w_1 \\
&w \stackrel H\rsg v \to u_1 \stackrel{H_1}\rsg v_1 \\
&w_1 \to w \stackrel H\rsg v \to u_1 \stackrel{H_1}\rsg_{w_1} v_1 \\
\text{Near-Hamiltonian paths:} \quad &w \stackrel H\rsg v \to u_1 \stackrel{H_1}\rsg_{v_1} w_1 \\
&w \stackrel H\rsg v \to u_1 \stackrel{H_1}\rsg_{w_1} v_1 \\
&w_1 \to u \stackrel H\rsg_w v \to u_1 \stackrel{H_1}\rsg_{w_1} v_1.
\end{align*}

\item Let $(u, v, w)$ be the outer face of a component $H$, which we assume is eventually strongly Hamiltonian. Join $H$ to a $K_{1, 2}$ as shown. 
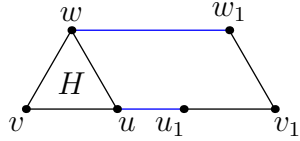
\begin{center}
\begin{asy}
size(4cm);
pair A = dir(90), B = dir(210), C = dir(-30), T = (3, 0), A1 = A+T, B1 = B+T, C1 = C+T;
draw(A--B--C--A^^A1--C1--B1); draw(A--A1^^C--B1, blue); 
dot("$w$", A, dir(90)); dot("$v$", B, dir(240)); dot("$u$", C, dir(-60));
dot("$w_1$", A1, dir(90)); dot("$u_1$", B1, dir(240)); dot("$v_1$", C1, dir(-60));
label("$H$", (A+B+C)/3);
\end{asy}
\end{center}
As discussed before, when this happens all the vertices in the game are either in $H$ or in $\{u_1, v_1, w_1\}$. Therefore it is enough to show that the above graph is Hamiltonian, which it clearly is:
\[
w \stackrel H\rsg u \to u_1 \to v_1 \to w_1 \to w.
\]
\end{enumerate}
This concludes the proof that Builder wins in the 2:1 game. 

\section{Proof of \Cref{thm:ham13}}
For all $i$, after Spoiler's $i$th move we will keep track of a particular subset $S_i$ of the vertices. If Spoiler moves first, Spoiler constructs a 3-cycle $(u, v, w)$ surrounding one vertex inside it. Put $S_1 = \{u, v, w\}$. 

If Builder moves first, drawing the edge $\{u,v\}$, Spoiler picks two isolated vertices $x,y$ and construct a 3-cycle $(u, x, y)$ surrounding $v$ inside it. Put $S_1 = \{u, x, y\}$.

Now suppose Spoiler has just made her $i$th move. Let $G_i$ be the current graph. We assume inductively that $G_i$ is comprised of isolated vertices and exactly one other component $H_i$ whose outer face is bounded by a 3-cycle $(u,v,w)$. We case on Builder's possible moves. 
\begin{itemize}
\item Suppose Builder draws the edge $\{x, u\}$, where $x$ is a vertex on the outer face of $H_i$. Spoiler then chooses vertices $a$ and $b$ on the outer face of $H_i$, provided they exist, and with her first two moves draws $\{x, v\}$ and $\{x, w\}$, surrounding $a$ and $b$ in the faces $(x, u, v)$ and $(x, u, w)$ as shown below. On Spoiler's third move, draw $\{x, a\}$; this is just a waiting move. Put $S_{i+1} = S_i \cup \{x\}$.  
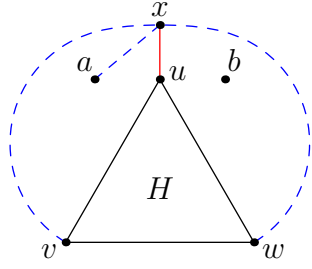
\begin{center}
\begin{asy}
size(4cm);
pair u = dir(90), v = dir(210), w = dir(-30), x = u + (0,0.5), a = u+(-0.6, 0), b = u+(0.6,0);
draw(u--v--w--u);
draw(u--x, red);
draw(v..a+(-0.3,0.3)..x..b+(0.3,0.3)..w, blue+dashed);
draw(x--a, blue+dashed);
dot("$u$", u, dir(20));
dot("$v$", v, dir(v));
dot("$w$", w, dir(w));
dot("$x$", x, dir(x));
dot("$a$", a, dir(a));
dot("$b$", b, dir(b));
label("$H$", (u+v+w)/3);
\end{asy}
\end{center}
\item Suppose Builder draws the edge $\{x, y\}$ where $x$ and $y$ are both on the outer face of $H_i$. Then Spoiler draws $\{x, u\}$, $\{x, v\}$, and $\{x, w\}$, enclosing $y$ in the face bounded by $(x, u, v)$ and another vertex $b$ in the face bounded by $(x, u, w)$.  Put $S_{i+1} = S_i \cup \{x\}$.
\end{itemize}
Now say Builder does neither of the above.
\begin{itemize}
\item If there is some isolated vertex $x$ on the outer face of $H_i$, then Spoiler draws $\{x, u\}$, $\{x, v\}$, and $\{x, w\}$, surrounding one vertex inside each of the faces bounded by $(x, u, v)$ and $(x, b, w)$. Put $S_{i+1} = S_i \cup \{x\}$.
\item Otherwise, Spoiler can play arbitrarily until the end of the game. In this case, put $S_j = S_i$ for all $j \geq i$. 
\end{itemize}
We now show that Spoiler's strategy works. Let $G$ be the graph at the end of the game, $k$ the number of moves Spoiler makes in total, and $T \subseteq H_k$ the subgraph induced by $S_k = S$. By construction $T$ is a planar triangulation and as such has $2|S|-4$ faces. All but at most two of the faces of $T$ encircle a vertex, so $3|S|-6 \leq n \leq 3|S|-4$. Then for $n$ large enough, there are more vertices inside faces of $T$ than there are vertices of $T$. It follows that deleting $S$ disconnects $G$ into more than $|S|$ components, so $G$ cannot be Hamiltonian. 

\section{Proof of \Cref{thm:degree}}
Let $v$ be the vertex which Builder nominates before the start of the game. 
\begin{definition}
Let $H$ be a graph, and let $F$ be a face of $H$ bounded by a 3-cycle, one of whose vertices is $v$. Denote by $\mathcal P(H, N, M)$ the graph $G$ defined by
\begin{itemize}
\item $V(G) = V(H) \sqcup V(N) \sqcup V(M)$, and
\item $E(G) = E(H) \sqcup \{\{u, v\} : u \in N\} \sqcup E(M)$, 
\end{itemize}
where $N$ is an independent set whose vertices lie entirely in $F$, and $M$ is a matching whose vertex set is disjoint from those of $H$ and $N$. If in addition to the above properties, $F$ contains isolated vertices, then we say that $F$ is an \textit{active face} with respect to $G$. (Remark that a given graph, such as a 3-cycle, may have more than one active face.) For brevity, we say ``$F$ is an active face" to mean that it is an active face with respect to $\mathcal P(H, N, M)$ for some $H,N,M$. 
\end{definition}
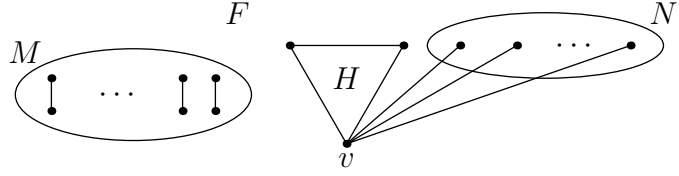
\begin{center}
\begin{asy}
size(9cm);
pair v = dir(-90), u = dir(150), w = dir(30), b = 2*w-u, a = (b+w)/2, c = 2*b-a, d = 2*c-b;
draw(v--u--w--v--a^^b--v--d);
draw(shift((b+c)/2)*xscale(1.8)*yscale(0.5)*unitcircle);
pair m1 = (-2,-0.5), m2 = (-2,0), m3 = (-2.5,-0.5), m4 = (-2.5,0), m5 = (-4.5,-0.5), m6 = (-4.5,0);
draw(m1--m2^^m3--m4^^m5--m6);
draw(shift((-3.25, -0.25))*xscale(1.8)*yscale(0.7)*unitcircle);
label("$\cdots$", (-3.5,-0.25));
label("$M$", m6+(-0.4, 0.4));
dot(m1^^m2^^m3^^m4^^m5^^m6);
dot(u^^w^^a^^b^^d);
// dot("1", a, dir(90)); dot("2", b, dir(90)); dot("$k$", d, dir(90));
dot("$v$", v, dir(-90));
label("$H$", (u+v+w)/3);
label("$\cdots$", c);
label("$N$", d+(0.5,0.5));
label("$F$", u+(-0.8,0.5));
\end{asy}
\end{center}

It is simple to check that if Builder (resp. Spoiler) moves first, then after Builder's third (resp. second) move, the resulting plane drawing will have at most two faces $F_1$ and $F_2$, and these satisfy the property that for each $i \in \{1, 2\}$, if $F_i$ contains any isolated vertices, then $F_i$ is an active face. 

\begin{definition}
Define a \textit{default move} as follows: Builder chooses an isolated vertex $u$ in an active face and draws $\{u, v\}$. 
\end{definition}
We distinguish two phases of Builder's strategy. In phase 1, Builder has at least one default move at his disposal. In phase 2, Builder no longer has any default moves. 

\subsection{Phase 1 strategy}
As long as Builder has a default move to play (i.e. the game is in phase 1), Builder has a strategy to inductively maintain the following invariant after each of his moves:
\begin{enumerate}
\item $\Omega(n)$ of all vertices in $V$ are either Builder-adjacent to $v$ or are isolated.
\item All isolated vertices lie in some active face. 
\item For each active face $F$ and corresponding subgraph $\mathcal P(H, N, M)$, if $M$ is nonempty, then $N$ is nonempty. 
\item There is at least one active face.
\end{enumerate}

We now present the strategy. If Spoiler does not play a move in an active face $A$, then Builder plays a default move. In what follows, we assume Spoiler moves in an active face with respect to $\mathcal P(H, N, M)$. 
\begin{itemize}
\item Suppose Spoiler joins $v$ to an isolated vertex $x$. This move does not break the invariant, so Builder can play any default move $\{v, y\}$. 
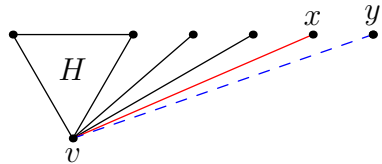
\begin{center}
\begin{asy}
size(5cm);
pair v = dir(-90), u = dir(150), w = dir(30), b = 2*w-u, a = (b+w)/2, c = 2*b-a, d = 2*c-b;
draw(v--u--w--v--a^^b--v);
draw(v--c, red);
draw(v--d, blue+dashed);
dot("$x$", c, dir(90));
dot("$y$", d, dir(90));
dot(u^^w^^a^^b^^d);
dot("$v$", v, dir(-90));
label("$H$", (u+v+w)/3);
\end{asy}
\end{center}
\item Suppose Spoiler draws $\{x, y\}$, where $x$ and $y$ are both isolated vertices. If there exists an isolated vertex $z$ in $A$ (i.e. $A$ is still an active face), then Builder plays $\{v, z\}$ for the sake of maintaining the invariant. Otherwise, Builder plays a default move. 
\item Suppose Spoiler draws $\{x, y\}$, where $x$ and $y$ are adjacent to $v$. This splits $A$ into two new faces $A_1$ and $A_2$. Let $S_1$ be the set of isolated vertices in $A_1$, and $S_2$ be the set of isolated vertices in $A_2$. Note that either $A_1$ or $A_2$ is an active face; without loss of generality, say $A_1$ is an active face, as shown in the diagram below. Let $N$ be the set of vertices adjacent to $v$ which lie in $A_1$. Then $A_1$ is active with respect to $\mathcal P(H', N, \emptyset)$, where $H'$ is the graph induced by the vertices $V \setminus (S_1 \cup N)$.  
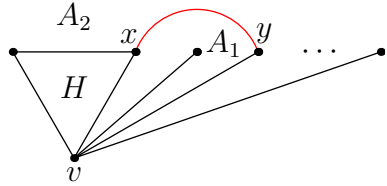
\begin{center}
\begin{asy}
size(5cm);
pair v = dir(-90), u = dir(150), w = dir(30), b = 2*w-u, a = (b+w)/2, c = 2*b-a, d = 2*c-b;
draw(v--u--w--v--a^^b--v--d);
draw(w..a+(0,0.6)..b, red);
dot(u^^w^^a^^b^^d);
dot("$v$", v, dir(-90));
dot("$A_1$", a, dir(20));
label("$A_2$", (u+w)/2 + (0,0.5));
dot("$x$", w, dir(110));
dot("$y$", b, dir(70));
label("$H$", (u+v+w)/3);
label("$\cdots$", c);
\end{asy}
\end{center}
At this point, the invariant is temporarily violated since all the vertices in $S_2$ no longer lie in an active face. But Builder can remedy this by making a move as shown below, splitting $A_2$ into two new faces $A_{2,1}$ and $A_{2,2}$ so that all vertices in $S_2$ now lie in $A_{2,2}$, and $A_{2,2}$ is an active face with respect to $\mathcal P(H'', \emptyset, \emptyset)$, where $H''$ is the graph induced by the closed neighborhood of $v$. 
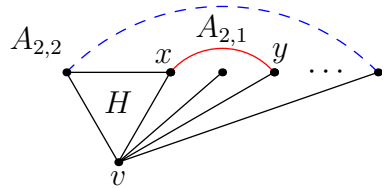
\begin{center}
\begin{asy}
size(5cm);
pair v = dir(-90), u = dir(150), w = dir(30), b = 2*w-u, a = (b+w)/2, c = 2*b-a, d = 2*c-b;
draw(v--u--w--v--a^^b--v--d);
draw(w..a+(0,0.4)..b, red);
draw(u..a+(0,1.1)..d, blue+dashed);
dot(u^^w^^a^^b^^d);
dot("$v$", v, dir(-90));
dot("$x$", w, dir(110));
dot("$y$", b, dir(70));
label("$H$", (u+v+w)/3);
label("$\cdots$", c);
label("$A_{2,1}$", a+(0,0.7));
label("$A_{2,2}$", u+(-0.5,0.5));
\end{asy}
\end{center}
\item Suppose Spoiler draws $\{x, y\}$, where $x$ is adjacent to $v$ and $y$ is an isolated vertex. Let $S$ be the set of isolated vertices in $A$. Builder can draw $\{v, y\}$ as shown below, splitting $A$ into two new faces $A_1$ and $A_2$ such that all vertices in $S$ lie in $A_2$, and $A_2$ is an active face with respect to $\mathcal P(H', \emptyset, \emptyset)$, where $H'$ is the graph induced by the closed neighborhood of $v$. Evidently this preserves the invariant. 
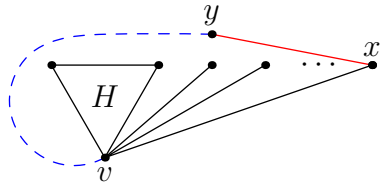
\begin{center}
\begin{asy}
size(5cm);
pair v = dir(-90), u = dir(150), w = dir(30), b = 2*w-u, a = (b+w)/2, c = 2*b-a, d = 2*c-b, y = a+(0,0.5);
draw(v--u--w--v--a^^b--v--d);
draw(y..(u+w)/2+(0,0.5)..u+(-0.5,0)..v, blue+dashed);
draw(d--y, red);
dot(u^^w^^a^^b^^d);
dot("$v$", v, dir(-90));
dot("$x$", d, dir(90));
dot("$y$", y, dir(90));
label("$H$", (u+v+w)/3);
label("$\cdots$", c);
\end{asy}
\end{center}
\item Finally, suppose Spoiler draws $\{x, y\}$, where $x$ lies in $M$ (the picture shows $y \in N$, but this need not be the case; perhaps $y \in M$ or otherwise). Let $S$ be the set of isolated vertices in $A$. By playing this move, Spoiler is threatening to surround $v$ in one move. To prevent this and simultaneously preserve the invariant, Builder plays $\{w, z\}$ as shown, where $w \in H$ and $z \in N$. Since $A$ is an active face and $M$ is nonempty, the invariant guarantees that such $z$ exists. 
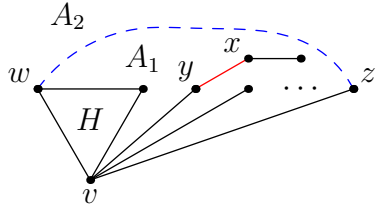
\begin{center}
\begin{asy}
size(5cm);
pair v = dir(-90), u = dir(150), w = dir(30), b = 2*w-u, a = (b+w)/2, c = 2*b-a, d = 2*c-b, m1 = b+(0,0.5), m2 = c+(0,0.5);
draw(v--u--w--v--a^^b--v);
draw(m1--m2);
draw(v--d);
draw(m1--a, red);
draw(u..w+(0,1)..a+(0,1)..m2+(0.2,0.2)..d, blue+dashed);
label("$\cdots$", c);
dot(m2);
dot("$x$", m1, dir(135));
dot("$w$", u, dir(135));
dot("$y$", a, dir(120));
dot("$z$", d, dir(45));
dot(u^^w^^a^^b^^d);
dot("$v$", v, dir(-90));
label("$H$", (u+v+w)/3);
label("$A_1$", w+(0,0.5));
label("$A_2$", u+(0.5,1.2));
\end{asy}
\end{center}
This move splits $A$ into two new faces $A_1$ and $A_2$. Builder makes sure that all vertices in $S$ lie in $A_2$, and that $A_2$ is an active face with respect to $\mathcal P(H', \emptyset, \emptyset)$, where $H'$ is the graph induced by $V(\mathcal P(H, N, M))$. Moreover, Builder that the new face $A_1$ does not surround any vertices of $M$ except $x$ and $y$ (if it were the case that $y \in M$). 
\end{itemize}
This concludes Builder's phase 1 strategy. 

\subsection{Phase 2 strategy}
Suppose now that the game has reached phase 2, i.e. it is Builder's turn, and Builder has run out of default moves. Let $G$ be the current graph. By item (2) of the phase 1 invariant, there are no more isolated vertices in $G$. Together with item (1) of the invariant, this implies that $\Omega(n)$ of all vertices in $V(G)$ are Builder-adjacent to $v$. That is, Builder's strategy in phase 2 no longer has to worry about the Builder-degree condition. 

Moreover, by examining the various cases of Builder's strategy in phase 1, we see that the only vertices which might possibly not be adjacent to $v$ in $G$ are either in $M$ or are ``localized" in the following sense: Builder's strategy guarantees that $G$ has a subgraph $H$ which satisfies the following property: for all vertices $x \in H$, if $x$ is not adjacent to $v$, then $x$ lies in a face $F$ incident with $v$ such that all but at most a constant number of the vertices in $F$ are adjacent to $v$. This ``worst case" scenario occurs in the last case, where Spoiler draws $\{x, y\}$, where $x$ and $y$ are both in $M$. Clearly this property ensures that every component of $G$ has constant diameter. 

Builder now plays as follows: if $M$ is empty, then $G$ is connected, and we are done (so Builder may move arbitrarily for the remainder of the game). If $H$ is not a triangulation, then Builder joins an arbitrary pair of points in $H$ and responds to Spoiler as in phase 1. If $H$ is a triangulation and $N$ is nonempty, then Builder joins a vertex of $y \in N$ and a vertex $x$ on the outer face of $H$ as shown below. 

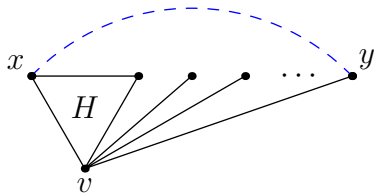
\begin{center}
\begin{asy}
size(5cm);
pair v = dir(-90), u = dir(150), w = dir(30), b = 2*w-u, a = (b+w)/2, c = 2*b-a, d = 2*c-b;
draw(v--u--w--v--a^^b--v--d);
draw(u..a+(0,1.1)..d, blue+dashed);
dot(u^^w^^a^^b^^d);
dot("$v$", v, dir(-90));
dot("$y$", d, dir(45));
dot("$x$", u, dir(135));
label("$H$", (u+v+w)/3);
label("$\cdots$", c);
\end{asy}
\end{center}

We may therefore assume that the active face under consideration is $\mathcal P(H, \emptyset, M)$, where $H$ is a triangulation and $M$ is nonempty. Builder chooses a vertex $x \in M$ and plays $\{v, x\}$, as shown below. 

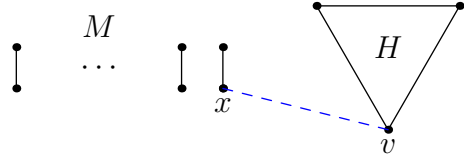
\begin{center}
\begin{asy}
size(6cm);
pair v = dir(-90), u = dir(150), w = dir(30);
draw(v--u--w--v);
pair m1 = (-2,-0.5), m2 = (-2,0), m3 = (-2.5,-0.5), m4 = (-2.5,0), m5 = (-4.5,-0.5), m6 = (-4.5,0);
draw(m1--m2^^m3--m4^^m5--m6);
label("$\cdots$", (-3.5,-0.25));
label("$M$", (-3.5, 0.25));
dot(m1^^m2^^m3^^m4^^m5^^m6);
dot(u^^w);
dot("$v$", v, dir(-90));
dot("$x$", m1, dir(-90));
draw(v--m1, blue+dashed);
label("$H$", (u+v+w)/3);
\end{asy}
\end{center}
Note that Spoiler cannot surround $v$ on her next move because there is no path (which does not contain $v$ itself) of length 2 in the outer face of $H$. Let $e$ be the edge that Spoiler plays. Let $H'$ be the connected component containing $v$. There are two cases.
\begin{itemize}
\item The outer face of $H'$ is a triangle. 
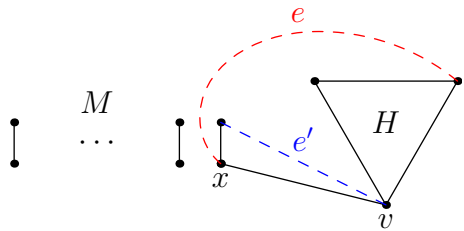
\begin{center}
\begin{asy}
size(6cm);
pair v = dir(-90), u = dir(150), w = dir(30);
draw(v--u--w--v);
pair m1 = (-2,-0.5), m2 = (-2,0), m3 = (-2.5,-0.5), m4 = (-2.5,0), m5 = (-4.5,-0.5), m6 = (-4.5,0);
draw(m1--m2^^m3--m4^^m5--m6);
label("$\cdots$", (-3.5,-0.25));
label("$M$", (-3.5, 0.25));
dot(m1^^m2^^m3^^m4^^m5^^m6);
dot(u^^w);
dot("$v$", v, dir(-90));
dot("$x$", m1, dir(-90));
draw(v--m1);
label("$H$", (u+v+w)/3);
draw(m1..((m2+m4)/2)..(u+(0, 0.6))..w, red+dashed);
label("$e$", u + (-0.2, 0.8), red);
draw(v--m2, blue+dashed);
label("$e'$", (v+m2)/2 + (0, 0.3), blue);
\end{asy}
\end{center}
Then no matter Spoiler's choice of $e$ (one example is shown above), Builder will be able to respond with $e'$ by joining two vertices in $H'$. 
\item The outer face of $H'$ is not a triangle. 
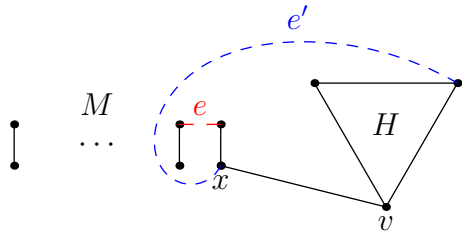
\begin{center}
\begin{asy}
size(6cm);
pair v = dir(-90), u = dir(150), w = dir(30);
draw(v--u--w--v);
pair m1 = (-2,-0.5), m2 = (-2,0), m3 = (-2.5,-0.5), m4 = (-2.5,0), m5 = (-4.5,-0.5), m6 = (-4.5,0);
draw(m1--m2^^m3--m4^^m5--m6);
label("$\cdots$", (-3.5,-0.25));
label("$M$", (-3.5, 0.25));
dot(m1^^m2^^m3^^m4^^m5^^m6);
dot(u^^w);
dot("$v$", v, dir(-90));
dot("$x$", m1, dir(-90));
draw(v--m1);
label("$H$", (u+v+w)/3);
draw(m2--m4, red+dashed);
label("$e$", (m2+m4)/2 + (0, 0.2), red);
draw(m1..(m3+(0,-0.2))..(m4+(-0.2,0.2))..(u+(-0.2,0.5))..w, blue+dashed);
label("$e'$", u+(-0.2, 0.8), blue);
\end{asy}
\end{center}
Then no matter how Spoiler moved (one example is shown above), Builder can play an edge $e'$ such that the outer face of the component containing $v$ is once again a triangle. If $e = \{a, b\}$ where $a,b \in M$, then Builder should play $e'$ in such a way that surrounds $a,b$. 
\end{itemize}
This guarantees the constant diameter condition. 

\section{Proof of \Cref{thm:apollonian}}
We first state a lemma about Apollonian networks. 
\begin{lemma}\label{apollonian}
Let $G$ be an Apollonian network, and fix a drawing. Choose a face of $G$ and replace it with an Apollonian network $H$, resulting in another graph $G'$. Then $G'$ is also an Apollonian network.
\end{lemma}
\begin{proof}
It suffices to show that any Apollonian network $G$ can be drawn such that the 3 vertices on the outer face are fixed throughout. Let $G_0, G_1, \ldots, G_n = G$ be a construction of $G$, where $G_i$ is obtained by adding a new vertex in a face of $G_{i-1}$, connecting it to each vertex of the face, and $G_0$ is a 3-cycle with $V(G_0) = \{a, b, c\}$. 

Say for some $i$ that the outer face of $G_{i-1}$ is $\{a, b, c\}$ but the outer face of $G_i$ is $\{a, b, d\}$, where $\{d\} = V(G_i) \setminus V(G_{i-1})$. Let's call such an index $i$ an \textit{illegal move}. Suppose that the sequence $G_0, \ldots, G_n$ has the minimum possible number of illegal moves. If this number is 0, then we are done; hence assume there is at least 1 illegal move. 

Now construct a new sequence $G_0', \ldots, G_{i-1}', G_i, \ldots, G_n = G$ where $G_0$ is a 3-cycle with $V(G_0') = \{a, b, d\}$ and $G_k'$ is defined by $V(G_k') = V(G_{k-1}) \cup \{d\}$ and $E(G_k') = E(G_{k-1}) \cup \{\{a, d\}, \{b, d\}, \{c, d\}\}$ for all $1 \leq k \leq i-1$. We thus obtain a new sequence with one fewer illegal move. 
\end{proof}
Now return to the theorem. Let $G_i$ be the graph after Spoiler's $i$th move. We exhibit a strategy of Spoiler so that for all $i$, every component of $G_i$ is a subgraph of an Apollonian network. Clearly this implies the theorem. Spoiler reacts to Builder's $(i+1)$th move as follows:
\begin{itemize}
\item If Builder joins two isolated points of $G_i$, then Spoiler likewise joins two isolated points. 
\item Let $H$ be a component of $G_i$ which we assume inductively is a subgraph of an Apollonian  network whose outer face is a 3-cycle $(u, v, w)$. Suppose Builder joins $H$ to an isolated point $x$, say via $\{x, u\}$. Then Spoiler draws $\{x, v\}$, surrounding $w$ (and no other points) in the process. 
\item Suppose Builder joins $H$ and an isolated edge $\{x, y\}$, say along $\{u, x\}$. Then Spoiler responds as pictured below at left, such that the newly formed face doesn't contain any other vertices. Let $H'$ be the new component, where $V(H') = V(H) \cup \{x, y\}$ and $E(H') = E(H) \cup \{\{x, y\}, \{y, u\}, \{u, x\}\}$. By induction we assume that Spoiler has a strategy such that, whenever Builder plays a move inside $H$, Spoiler can respond such that $H$ remains a subgraph of an Apollonian network. Moreover, note that there is a unique planar triangulation on 5 vertices, and it is an Apollonian network. It follows  by \Cref{apollonian} that Builder has a strategy to ensure that $H'$ is always a subgraph of an Apollonian network.  
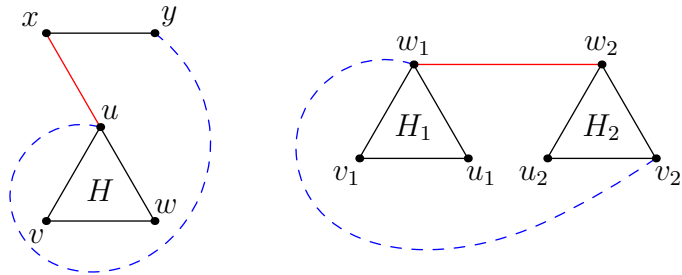
\begin{center}
\begin{asy}
picture l = new picture; picture ll = new picture;
size(9cm);
pair A = dir(90), B = dir(210), C = dir(-30), T = (3, 0), A1 = A+T, B1 = B+T, C1 = C+T;
pair D = 2*A - C, E = 2*A - B;
draw(ll, A--B--C--A^^D--E); draw(ll, A--D, red); draw(ll, E..C+(0.5,0)..B+(-0.5,0)..A, blue+dashed);
dot(ll, "$u$", A, dir(60)); dot(ll, "$v$", B, dir(240)); dot(ll, "$w$", C, dir(60)); dot(ll, "$x$", D, dir(135)); dot(ll, "$y$", E, dir(45));
draw(l, A..B+(-1, 0)..B1+(0, -1)..C1, blue+dashed);
draw(l, A--B--C--A^^A1--B1--C1--A1); draw(l, A--A1, red);
dot(l, "$w_1$", A, dir(90)); dot(l, "$v_1$", B, dir(240)); dot(l, "$u_1$", C, dir(-60));
dot(l, "$w_2$", A1, dir(90)); dot(l, "$u_2$", B1, dir(240)); dot(l, "$v_2$", C1, dir(-60));
label(l, "$H_1$", (A+B+C)/3); label(l, "$H_2$", (A1+B1+C1)/3);
label(ll, "$H$", (A+B+C)/3);
add(shift(5,1)*l); add(shift(0,0)*ll);
\end{asy}
\end{center}
\item Suppose Builder joins $H_1$ and $H_2$ along $\{w_1, w_2\}$, where $H_1$ and $H_2$ are both subgraphs of Apollonian networks whose outer faces are the 3-cycles $(u_1, v_1, w_1)$ and $(u_2, v_2, w_2)$. Then Spoiler responds as pictured above at right, such that the enclosed regions do not contain any other vertices. Let $H$ be the new component; that is, $V(H) = V(H_1) \cup V(H_2)$ and $E(H) = E(H_1) \cup E(H_2) \cup \{\{w_1, v_2\}, \{v_2, w_2\}, \{w_2, w_1\}\}$. For $i \in \{1, 2\}$, inductively assume that Spoiler has a strategy such that, if Builder plays a move in $H_i$, then Spoiler can respond such that $H_i$ remains a subgraph of an Apollonian network. So it suffices to specify how Spoiler should respond if Builder draws an edge between a vertex in $H_1$ and a vertex in $H_2$. Note that there are two nonisomorphic planar triangulations on 6 vertices, pictured in \Cref{fig:planar6}. 
\begin{figure}[ht]
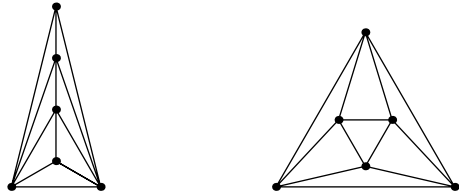

\begin{asy}
size(6cm);
picture l = new picture; picture r = new picture;
pair A = dir(90), B = dir(210), C = dir(-30), D = (0, 0), E = 2*A, F = 3*A;
pair X = 2*A, Y = 2*B, Z = 2*C, X1 = -0.3*X, Y1 = -0.3*Y, Z1 = -0.3*Z;
draw(l, B--D--C--B--A--C--E--B--F--E--A--D--C--F);
dot(l, A^^B^^C^^D^^E^^F);
draw(r, X--Y--Z--X1--Y--Z1--Y1--X1--Z1--X--Y1--Z--X);
dot(r, X^^Y^^Z^^X1^^Y1^^Z1);
add(shift(0,0)*l); add(shift(6,0.5)*r);
\end{asy}
\centering
\caption{The two planar triangulations on 6 vertices}
\label{fig:planar6}
\end{figure}
The one at left that is an Apollonian network has the distinguishing feature of containing a $K_4$. Therefore as long as Spoiler can respond to Builder's moves in $H$ such that a $K_4$ appears, then the subgraph induced by $\{v_1, w_1, u_1, v_2, w_2, u_2\}$ will remain a subgraph of an Apollonian network. It's not hard to see that this can be done. Combining with the inductive hypothesis, \Cref{apollonian} then ensures that, at this point, Spoiler has a strategy such that $H$ will remain a subgraph of an Apollonian network. 

\item If Builder draws an edge between two vertices in the same component $H$, then Spoiler should likewise play within $H$. If $H$ is eventually an Apollonian network, then Spoiler can draw an arbitrary edge. 
\end{itemize}
Spoiler is therefore able to maintain the inductive invariant throughout the course of the game, proving the theorem. 

\section{Proof of \Cref{thm:circle}}
In this section, all tangencies are \textit{external} tangencies. 
\begin{definition}\label{def:dist}
Let $\mathcal C = \{\Omega_1, \ldots, \Omega_n\}$ be a circle packing. The \textit{degree} of a circle $\Omega_i$ (with respect to $\mathcal C$) is the degree of $i$ in $\mathbf G(\mathcal C)$. 

Given a circle packing $\mathcal C$ and circles $\Omega_1, \Omega_2 \in \mathcal C$ centered at $O_1, O_2$ with radii $r_1, r_2$ respectively, the \textit{distance} between $\Omega_1$ and $\Omega_2$ is defined to be the quantity $O_1O_2 - r_1 - r_2$ (where $O_1O_2$ denotes the distance between $O_1$ and $O_2$). 
\end{definition}

\begin{definition}
Let $\omega_0, \ldots, \omega_{n-1}$ and $\Omega$ be circles. We say that $\{\omega_i\}_{i=0}^{n-1}$ \textit{surrounds} $\Omega$ if:
\begin{itemize}
    \item $\mathbf G\left(\{\omega_i\}_{i=0}^{n-1}\right)$ is an induced cycle: for all $0 \leq i < j \leq n-1$, $\omega_i$ and $\omega_j$ are tangent if and only if $j \equiv i + 1 \pmod n$; and
    \item $\omega_i$ is tangent to $\Omega$ for all $0 \leq i \leq n-1$.
\end{itemize}
Accordingly, given a circle packing $\mathcal C$ and $\Omega \in \mathcal C$, we say that $\Omega$ is an \textit{inner circle} in $\mathcal C$ if there exist $\{\omega_i\}_{i=0}^{n-1} \subseteq \mathcal C$ which surrounds $\Omega$. Say that $\Omega$ is an \textit{outer circle} in $\mathcal C$ if $\Omega$ touches the outer face of $\mathcal C$. 
\end{definition}

\begin{figure}[h]
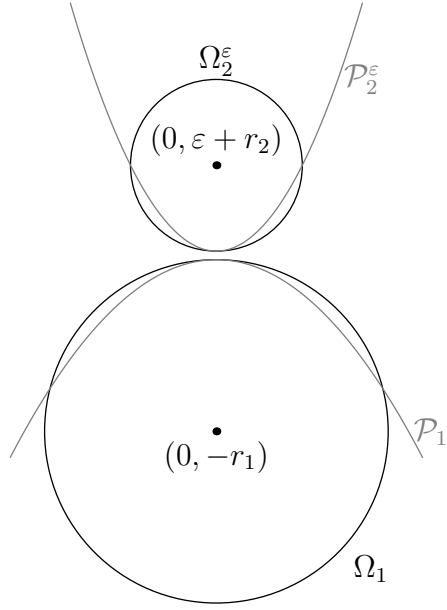

\begin{asy}
size(8cm);
import graph;
real r1 = 2, r2 = 1, eps = 0.1, del1 = 0.4, del2 = 1;
pair O1 = (0, -r1), O2 = (0, eps + r2);
dot("$(0, -r_1)$", O1, S);
dot("$(0, \varepsilon + r_2)$", O2, N);
draw(circle(O1, r1));
draw(circle(O2, r2));
real P1(real x) { return -del1 * x * x; }
real P2(real x) { return eps + del2 * x * x; }
draw(graph(P1, -2.4, 2.4), gray);
draw(graph(P2, -1.7, 1.7), gray);
label("$\Omega_1$", 0.9 * (r1, -2 * r1));
label("$\Omega_2^\varepsilon$", 1.1 * (0, eps + 2 * r2));
label("$\mathcal P_1$", (2.5, -r1), gray);
label("$\mathcal P_2^\varepsilon$", (1.7, eps + r2 + 1), gray);
\end{asy}
\centering
\caption{Filling the gap between nearly touching circles}
\label{fig:width}
\end{figure}

\begin{lemma}\label{lem:width}
Let $r_1, r_2 > 0$. For all $d \in \mathbb N$, there exists $\varepsilon > 0$ such that if 
\begin{itemize}
    \item $\Omega_1$ and $\Omega_2$ are internally disjoint circles of radii $r_1$ and $r_2$, and $\varepsilon$ is the distance between $\Omega_1$ and $\Omega_2$; and
    \item $\{\omega_i\}_{i=0}^{n-1}$ surrounds $\Omega_1$, and the face bounded by $\{\omega_i\}_{i=0}^{n-1}$ does not contain $\Omega_2$,
\end{itemize}
then $n \geq d$. We denote $\varepsilon = \text{width}(r_1, r_2, d)$.
\end{lemma}
\begin{proof}
Fix $r_1, r_2 > 0$ and $d \in \mathbb N$. Let $\Omega_1$ be centered at $(0, -r_1)$ with radius $r_1$ which passes through $(0,0)$, and for $\varepsilon > 0$, denote by $\Omega_2^\varepsilon$ the circle centered at $(0, r_2 + \varepsilon)$ which passes through $(0, \varepsilon)$. We approximate these circles by parabolas near 0. Fix constants $\delta_1, \delta_2, a > 0$ such that, over the interval $(-a, a)$, the curve $x \mapsto -\delta_1x^2$ lies below the upper half of $\Omega_1$ and the curve $x \mapsto \varepsilon + \delta_2x^2$ lies above the lower half of $\Omega_2^\varepsilon$, as shown in \Cref{fig:width}. Let $\mathcal P_1$ and $\mathcal P_2^\varepsilon$ denote these two parabolas, respectively.

Now using the fact that
\[
\frac{b}{2(\varepsilon + (\delta_1 + \delta_2)b^2)} \to \infty \quad \text{as} \quad (b, \varepsilon) \to (0, 0),
\]
we obtain $M > 0$ such that for all $\|(b, e)\| < M$, it holds that 
\[
\frac{b}{2(\varepsilon + (\delta_1 + \delta_2)b^2)} > d.
\]
Fix some $b , \varepsilon > 0$ with $\|(b, \varepsilon)\| < M$. Without loss of generality, assume $\omega_0$ intersects the line $x = 0$. Let $x_0 = \min\{a, b\}$, and let $\ell$ be the vertical line $x = x_0$. Then $\omega_k$ intersects $\ell$ for some $0 \leq k \leq n-1$. For all $0 \leq i \leq k-1$, since the $x$-coordinate of the center of $\omega_i$ is at most $x_0$, the radius of $\omega_i$ is at most $\varepsilon + (\delta_1 + \delta_2)x_0^2$, which is the vertical distance between $\mathcal P_1$ and $\mathcal P_2^\varepsilon$. It follows that
\[
k \geq \frac{x_0}{2(\varepsilon + (\delta_1 + \delta_2)x_0^2)} > d,
\]
as we wanted. 
\end{proof}

The following lemma is an immediate consequence of Descartes' circle theorem. 
\begin{lemma}\label{lem:descartes}
Suppose $\Omega$ lies in the face bounded by the mutually tangent circles $\omega_1, \omega_2, \omega_3$, as shown in \Cref{fig:descartes}. Then
\[
\frac1r > \frac{1}{r_1} + \frac{1}{r_2} + \frac{1}{r_3},
\]
where $r,r_1,r_2,r_3$ are the radii of $\Omega, \omega_1, \omega_2, \omega_3$ respectively. 
\end{lemma}

\begin{figure}
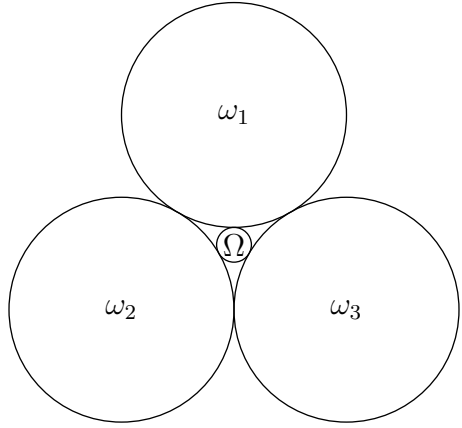

\begin{asy}
size(6cm);
import geometry;
pair A = dir(90), B = dir(210), C = dir(330), O = (0, 0), D = (B+C)/2, E = (C+A)/2, F = (A+B)/2;
real r = abs(A - B)/2; 
draw(circle(A, r)^^circle(B, r)^^circle(C, r)^^circle(O, 1 - r));
label("$\omega_1$", A);
label("$\omega_2$", B);
label("$\omega_3$", C);
label("$\Omega$", O);
\end{asy}
\centering
\caption{Bounds on radii of pairwise tangent circles}
\label{fig:descartes}
\end{figure}

\begin{lemma}\label{lem:bigcircle}
Let $d \in \mathbb N$ and let $\mathcal C \neq \emptyset$ be a circle packing with $\Omega \in \mathcal C$. Suppose $\Omega$ is not an inner circle in $\mathcal C$. Then there exists a circle $\Gamma = \Gamma(\mathcal C, \Omega, d)$ such that the following hold:
\begin{enumerate}
    \item $\mathcal C \cup \{\Gamma\}$ is a circle packing.
    \item For any circle packing $\mathcal D \supseteq \mathcal C \cup \{\Gamma\}$, the following hold: if $\Omega$ is an inner circle in $\mathcal D$, then the degree of $\Omega$ is at least $d$; if $\Gamma$ is an inner circle in $\mathcal D$, then the degree of $\Gamma$ is at least $d$. 
\end{enumerate}
\end{lemma}
\begin{proof}
For $P \in \Omega$, say that $P$ is a \textit{gap point} if either
\begin{itemize}
    \item there does not exist a minor arc $\wideparen{T_1T_2}$ such that $P \in \wideparen{T_1T_2}$ and there exist circles $\omega_1, \omega_2 \in \mathcal C$ which are tangent to each other and to $\Omega$ at $T_1$ and $T_2$; or
    \item there exists $\omega \in \mathcal C$ tangent to $\Omega$ at $P$.
\end{itemize}
Denote by $\text{GapPts} \subseteq \Omega$ the set of gap points, and note that GapPts is closed. For example, in \Cref{fig:gappoints}, the gap points of $\Omega$ are marked in bold. 
\begin{figure}
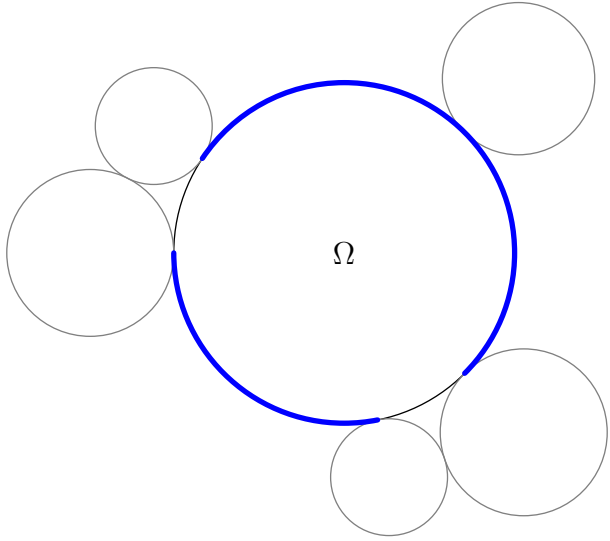

\begin{asy}
size(8cm);
import geometry;
real rot = 135;
pair O = (0, 0), P = (-4, 0), Q = (-3, 2);
pair I = (-2.68474, 0.812874);
pair D = (I + reflect(P, O) * I)/2;
pair E = (I + reflect(Q, O) * I)/2;
pair F = (I + reflect(P, Q) * I)/2;
pair R = rotate(rot) * P, S = rotate(rot) * Q;
real r0 = abs(O - E), r1 = abs(Q - F), r2 = abs(P - D), r3 = 1.2;
pair T = (r0 + r3, r0 + r3) * (1 / sqrt(2));
draw(circle(O, r0));
draw(circle(Q, r1)^^circle(P, r2)^^circle(S, r1)^^circle(R, r2)^^circle(T, r3), gray);
label("$\Omega$", O);
draw(arc(O, r0, aTan(R.y/R.x), aTan(Q.y/Q.x) + 180), linewidth(2) + blue);
draw(arc(O, r0, aTan(S.y/S.x), aTan(P.y/P.x) + 180, CW), linewidth(2) + blue);
\end{asy}
\centering
\caption{Gap points of $\Omega$}
\label{fig:gappoints}
\end{figure}

For each $P \in \text{GapPts}$, let $R(P) \subseteq \mathbb R$ be the set of all $r \geq 0$ for which there exists a circle $\omega$ of radius $r$ tangent to $\Omega$ at $P$ such that $\mathcal C \cup \{\omega\}$ is a circle packing. Let $r_\Omega$ be the radius of $\Omega$. 

We split into two cases. Suppose first that there is some $P \in \text{GapPts}$ such that either $R(P)$ is unbounded or $\max R(P) > r_\Omega$. This implies that for $\varepsilon > 0$ sufficiently small, there exists a circle $\Gamma$ of radius $r_\Omega$ such that $\mathcal C \cup \{\Gamma\}$ is a circle packing and the distance between $\Omega$ and $\Gamma$ is $\varepsilon$. In particular, choose $\varepsilon$ to be smaller than $\text{width}(r_\Omega, r_\Omega, d)$. 

We claim that $\Gamma$ satisfies item (2). Let $\mathcal D \supseteq \mathcal C \cup \{\Gamma\}$ be a circle packing, and suppose $\Omega$ is an inner circle in $\mathcal D$ (the argument is the same if instead $\Gamma$ is an inner circle in $\mathcal D$). Let $\{\omega_i\}_{i=0}^{n-1} \subseteq \mathcal D$ surround $\Omega$. Let $F$ be the face bounded by $\{\omega_i\}_{i=0}^{n-1}$. 
\begin{itemize}
    \item If $\Gamma$ does not lie in $F$, then we are done by \Cref{lem:width}.
    \item If $\Gamma$ lies in $F$, then in particular $\Gamma$ lies in the face bounded by $\omega_i$, $\omega_{i+1}$, and $\Omega$ for some $i$ (indices taken modulo $n$). This is impossible by \Cref{lem:descartes} and the fact that $\Gamma$ and $\Omega$ have the same radius. 
\end{itemize}

Now suppose that $R(P)$ is bounded and $\max R(P) \leq r_\Omega$ for all $P \in \text{GapPts}$. Define 
\[
T = \argmax_{P \in \text{GapPts}} (\max R(P)) \in \text{GapPts}, \quad r_* = \max R(T), \quad \eta = \frac{r_*^2}{r_* + r_\Omega}.
\]
Note that $T$ exists since $\text{GapPts} \subseteq \Omega$ is closed. By definition of $T$, there exists a circle $\Gamma'$ with center $O$ and radius $r_*$ tangent to $\Omega$ at $T$ such that $\mathcal C \cup \{\Gamma'\}$ is a circle packing. Let $\Gamma$ be the circle whose center lies on the segment $\overline{OT}$ with radius $r_\Gamma = r_* - \eta$ such that the distance between $\Gamma$ and $\Omega$ is $\varepsilon = \min\{\eta, \text{width}(r_\Omega, r_\Gamma, d)\}$. Note that $\Gamma$ is contained entirely inside $\Gamma'$, so $\mathcal C \cup \{\Gamma\}$ is a circle packing. 

It suffices to show that $\Gamma$ satisfies item (2). Let $\mathcal D \supseteq \mathcal C \cup \{\Gamma\}$ be a circle packing. Suppose first that $\Gamma$ is an inner circle in $\mathcal D$. Let $\{\omega_i\}_{i=0}^{n-1} \subseteq \mathcal D$ surround $\Gamma$. Let $F$ be the face bounded by $\{\omega_i\}_{i=0}^{n-1}$. As before, if $\Omega$ does not lie in $F$, then we are done by \Cref{lem:width}. Otherwise, $\Omega$ must lie in the face bounded by $\omega_i$, $\omega_{i+1}$, and $\Gamma$ for some $i$. But this violates \Cref{lem:descartes} as
\[
r_\Gamma = r_* - \eta < r_* \leq r_\Omega,
\]
by the assumption that $\max R(T) \leq r_\Omega$. 

Suppose instead that $\Omega$ is an inner circle in $\mathcal D$. Let $\{\omega_i\}_{i=0}^{n-1} \subseteq \mathcal D$ surround $\Omega$. Let $F$ be the face bounded by $\{\omega_i\}_{i=0}^{n-1}$. Once again, if $\Omega$ does not lie in $F$, then we are done by \Cref{lem:width}. So suppose $\Omega$ lies in $F$. Then $\Gamma$  must lie in the face bounded by $\omega_i$, $\omega_{i+1}$, and $\Gamma$ for some $i$. Since $T \in \text{GapPts}$, at least one of $\omega_i$ and $\omega_{i+1}$ is not in $\mathcal C$. Without loss of generality, assume $\omega_i \notin \mathcal C$. Let $\omega_i$ be tangent to $\Omega$ at $T_i$. From $T \in \text{GapPts}$, it follows that $T_i \in \text{GapPts}$. By \Cref{lem:descartes}, we have
\[
\frac{1}{r_\Gamma} > \frac{1}{r_\Omega} + \frac{1}{r_{\omega_i}} + \frac{1}{r_{\omega_{i+1}}} > \frac{1}{r_\Omega} + \frac{1}{r_{\omega_i}},
\]
where $r_\omega$ is the radius of $\omega$, etc., which rearranges to
\[
r_{\omega_i} > \left(\frac{1}{r_\Gamma} - \frac{1}{r_\Omega}\right)^{-1} =\left(\frac{1}{r_* - \eta} - \frac{1}{r_\Omega}\right)^{-1} = \left(\frac{1}{r_* - \frac{r_*^2}{r_* + r_\Omega}} - \frac{1}{r_\Omega}\right)^{-1} = r_*,
\]
which contradicts the definition of $r_*$. 
\end{proof}

We are now ready to present Spoiler's strategy. Let $H$ be a planar triangulation with $|V(H)| \geq 3$, and suppose Builder's goal is to construct $H$. Let $\Delta$ be the maximum degree in $H$. Suppose Builder places a circle $\Omega$, and the current circle packing is $\mathcal C$. If $\Omega$ is an inner circle in $\mathcal C$, then Spoiler plays an arbitrary circle not adjacent to anything else. Otherwise, Spoiler plays a circle $\Gamma = \Gamma(\mathcal C, \Omega, \Delta + 1)$ as specified by \Cref{lem:bigcircle}. We now show this strategy works, in the sense that if Builder can successfully construct $H$, then $H$ must be an Apollonian network. 

Suppose Builder has a strategy to build $H$. This means there is a sequence of moves
\[
\Omega_1, \Gamma_1, \Omega_2, \Gamma_2, \ldots, \Omega_N, \Gamma_N,
\]
such that $H$ is a subgraph of $\mathbf G(\mathcal C_N)$, where:
\begin{itemize}
\item $\mathcal C_i = \{\Omega_1, \Gamma_1, \ldots, \Omega_i, \Gamma_i\}$ is the position after both players have played $i$ moves,
\item $\{\Omega_i\}_{i=1}^N$ are the circles played by Builder,
\item $\{\Gamma_i\}_{i=1}^N$ are the circles played by Spoiler, and
\item each $\Gamma_i$ is played according to Spoiler's strategy as described above in response to $\Omega_i$. 
\end{itemize}
\begin{definition}
Let $\mathcal D \subseteq \mathcal C_N$ be such that $\mathbf G(\mathcal D)$ is isomorphic to $H$. Since $H$ is a triangulation, $\mathcal D$ has exactly 3 outer circles. The rest are inner circles. For each $\omega \in \mathcal C_N$, say that $\omega$ is \textit{relevant} if $\omega \in \mathcal D$; otherwise, say that $\omega$ is \textit{irrelevant}. 
\end{definition}

The proof of correctness proceeds with a series of claims. 
We begin by proving a lemma about planar graph drawings, from which the first claim will immediately follow. 
\begin{lemma}\label{lem:planarcycle}
Let $G$ be a plane graph drawing. Let $w \in G$ and $U = u_0, \ldots, u_{n-1}$ be an induced cycle with each $u_i$ adjacent to $w$. Suppose $w$ lies in the bounded region of $\mathbb R^2 \setminus U$. Then for any cycle $V = v_0, \ldots, v_{m-1}$ with each $v_i$ adjacent to $w$ and for which $w$ lies in the face bounded by $V$, we have $U \subseteq V$. 
\end{lemma}
\begin{proof}
Fix such a cycle $V = v_0, \ldots, v_{m-1}$. For $0 \leq i \leq n-1$, denote by $F_i$ be the union of the singleton $\{u_i\}$ and the interior of the face bounded by $w$, $u_i$, and $u_{i+1}$, indices taken modulo $n$. From the condition that $w$ lies in the face bounded by $U$, it follows that for each $0 \leq j \leq m-1$, there is some $0 \leq i \leq n-1$ such that $v_j$ lies in $F_i$. Thus we have a function $f : \mathbb Z/m\mathbb Z \to \mathbb Z/n\mathbb Z$ defined by the property that $v_j$ lies in $F_{f(j)}$. From the condition that $w$ lies in the face bounded by $V$, it follows that $f$ is surjective. Moreover, 
\[
f(j + 1) \in \{f(j) - 1, f(j), f(j) + 1\}
\]
for each $j$. Together with surjectivity, this implies that for every $i$, there is some $j$ such that $\{f(j), f(j + 1)\} = \{i, i + 1\}$. But this can only happen if one of $v_i$ and $v_{i+1}$ coincides with $u_{i+1}$. 
\end{proof}

\begin{claim}\label{claim:layers}
Fix $1 \leq i \leq N$. Suppose $\Omega_i$ is an inner circle in $\mathcal D$. Suppose $\{\omega_j\}_{j=0}^{n-1} \subseteq \mathcal C_i$ surrounds $\Omega_i$. Then each $\omega_j$ is relevant. 
\end{claim}
\begin{proof}
For $0 \leq j \leq n-1$, let $u_j$ denote the center of $\omega_j$ and $w$ the center of $\Omega_i$. In the planar graph drawing $\mathbf G(\mathcal C_N)$, the vertex $w$ is adjacent to every vertex in the induced cycle $U = u_0, \ldots, u_{n-1}$, and $w$ lies in the face bounded by $U$. Since $\Omega_i$ is an inner circle in $\mathcal D$, there is another set of circles $\{\gamma_j\}_{j=0}^{m-1} \subseteq \mathcal D \subseteq \mathcal C_N$ which surrounds $\Omega_i$. For each $0 \leq j \leq m-1$, let $v_j$ denote the center of $\gamma_j$. Then in $\mathbf G(\mathcal C_N)$, the vertex $w$ is adjacent to every vertex in the cycle $V = v_0, \ldots, v_{m-1}$, and $w$ lies in the face bounded by $V$. By \Cref{lem:planarcycle}, $U \subseteq V$, which implies that $\omega_j$ is relevant for all $0 \leq j \leq n-1$. 
\end{proof}

Our central claim, which is a corollary of \Cref{lem:bigcircle}, isolates the essential combinatorial properties of Spoiler's strategy. 
\begin{claim}\label{claim:main}
Fix $1 \leq i \leq N$. The following hold. 
\begin{enumerate}
    \item $\Omega_i$ is either a relevant inner circle in $\mathcal C_i$, an outer circle in $\mathcal D$, or irrelevant. 
    \item $\Gamma_i$ is either an outer circle in $\mathcal D$ or irrelevant. 
\end{enumerate}
\end{claim}
\begin{proof}
We begin by proving item (1). Suppose $\Omega_i$ is relevant, not an inner circle in $\mathcal C_i$, and not an outer circle in $\mathcal D$. Then $\Omega_i$ is an inner circle in $\mathcal D \supseteq \mathcal C_i \cup \{\Gamma_i\}$, meaning that there exists some $\{\omega_j\}_{j=0}^{n-1} \subseteq \mathcal D$ which surrounds $\Omega_i$. By \Cref{lem:bigcircle}, $n \geq \Delta + 1$. However, by \Cref{claim:layers}, each $\omega_j$ is relevant, so the maximum degree of $\mathbf G(\mathcal D)$ exceeds that of $H$, contradiction. This proves item (1). 

The proof of item (2) proceeds along similar lines. Suppose $\Gamma_i$ is an inner circle in $\mathcal D$, meaning that there exists some $\{\omega_j\}_{j=0}^{n-1} \subseteq \mathcal D$ which surrounds $\Gamma_i$. By \Cref{lem:bigcircle}, $n \geq \Delta + 1$, and finish as before. 
\end{proof}

Next we show that, over the course of the game, the 3 outer circles in $\mathcal D$ must be placed before any inner circles of $\mathcal D$. 
\begin{claim}\label{claim:outerfirst}
Let $\Phi$ be the first circle in the sequence $\Omega_1, \Gamma_1, \ldots, \Omega_N, \Gamma_N$ which is an inner circle in $\mathcal D$. The following hold. 
\begin{enumerate}
    \item $\Phi = \Omega_i$ for some $i$. 
    \item Any circle which occurs earlier than $\Phi$ in the sequence is either one of the 3 outer circles of $\mathcal D$ or irrelevant. 
\end{enumerate}
\end{claim}
\begin{proof}
By item (2) of \Cref{claim:main}, $\Phi$ must be $\Omega_i$ for some $i$. Since $\Omega_i$ is an inner circle in $\mathcal D$, by item (1) of \Cref{claim:main}, it must be an inner circle in $\mathcal C_i$, so there exists some $\{\omega_j\}_{j=0}^{n-1} \subseteq \mathcal C_i$ which surrounds $\Omega_i$. By \Cref{claim:layers}, each $\omega_j$ is relevant. But $\Omega_i$ is the \textit{first} relevant inner circle in the sequence, so each $\omega_j$ must be outer circles in $\mathcal D$. Since the induced cycle $\{\omega_j\}_{j=0}^{n-1}$ consists of at least 3 circles, it in fact must consist of exactly 3 circles, these being the outer circles of $\mathcal D$. 
\end{proof}
For each $1 \leq i \leq N$, define $\mathcal D_i = \mathcal D \cap \mathcal C_i$. In particular, $\mathcal D_N = \mathcal D$. 
\begin{claim}\label{claim:apollonian}
For all $i$, $\mathbf G(\mathcal D_i)$ is either a subgraph of $K_3$ or an Apollonian network. 
\end{claim}
\begin{proof}
We induct on $i$. Let $i_0$ be the smallest index such that $\Omega_{i_0}$ is an inner circle in $\mathcal D$. By \Cref{claim:outerfirst}, $\mathbf G(\mathcal D_{i_0 - 1}) \cong K_3$. Now let $i \geq i_0 - 1$, and assume inductively that $\mathbf G(\mathcal D_i)$ is an Apollonian network. We aim to show that $\mathbf G(\mathcal D_{i+1})$ is an Apollonian network. Since $i + 1 > i_0$ so all 3 outer circles in $\mathcal D$ have already been placed at this point, by item (2) of \Cref{claim:main}, $\Gamma_{i+1}$ is irrelevant. By item (1) of \Cref{claim:main}, $\Omega_{i+1}$ is either a relevant inner circle in $\mathcal C_{i+1}$ or irrelevant. If $\Omega_{i+1}$ is irrelevant, then $\mathcal D_{i+1} = \mathcal D_i$, and we are done. 

Assume, then, that $\Omega_{i+1}$ is a relevant inner circle in $\mathcal C_{i+1}$, so that $\mathcal D_{i+1} = \mathcal D_i \cup \{\Omega_{i+1}\}$. Let $U = \{\omega_j\}_{j=0}^{n-1} \subseteq \mathcal C_{i+1}$ surround $\Omega_{i+1}$. By \Cref{claim:layers}, $\omega_j$ is relevant or all $j$, i.e. $U \subseteq \mathcal D_{i+1}$. In fact, $U \subseteq \mathcal D_i$ since $\mathcal D_{i+1} = \mathcal D_i \cup \{\Omega_{i+1}\}$. Since $\mathbf G(U) \subseteq \mathbf G(\mathcal D_i)$ is an induced cycle and Apollonian networks are triangulations, it follows that $\mathbf G(U) \cong K_3$. Therefore $\mathbf G(\mathcal D_{i+1})$ is an Apollonian network. 
\end{proof}

In particular, \Cref{claim:apollonian} implies that $\mathbf G(\mathcal D_N)$ is either a subgraph of $K_3$ or an Apollonian network, which proves \Cref{thm:circle}. 

\section{Proof of \Cref{thm:biasedcircle}}
We first state a classic result about the box game, introduced by Chvátal and Erdős \cite{CE78, HLV87, HKSS14}. Let $\mathcal D(m, n)$ denote the hypergraph with $m$ disjoint edges, each of size $n$. Let $\varepsilon \in (0, 1)$. The \textit{$(1+\varepsilon) : 1$ biased $(m, n)$ box game}, is defined as follows. Two players, Maker and Breaker, take turns occupying a vertex of $\mathcal D(m, n)$. Every $\floor{1/\varepsilon}$ moves, Maker is given one additional move. Maker's goal is to fully occupy an edge of $\mathcal D(m, n)$; Breaker's goal is to prevent this. 
\begin{lemma}\label{lem:boxgame}
Fix $n \geq 1$ and $\varepsilon \in (0, 1)$. If $m \geq \floor{1/\varepsilon + 1}^{n-1}$ then Maker wins the $(1+\varepsilon) : 1$ biased $(m, n)$ box game. 
\end{lemma}
\begin{proof}
Induct on $n$. If $n=1$, then Maker wins on his very first turn. Now suppose $n \geq 2$, and let $q = \floor{1/\varepsilon}$. In the first $q(q+1)^{n-2}$ turns, Maker puts one mark in each of the $(q+1)^{n-1}$ edges. No matter how Breaker responded to those moves, there must be at least $(q+1)^{n-1} - q(q+1)^{n-2} = (q+1)^{n-2}$ edges which have one vertex occupied by Maker and none occupied by Breaker. Maker now plays the $(1+\varepsilon):1$ biased $(\floor{1/\varepsilon + 1}^{n-2}, n-1)$ box game on those edges, and by induction, wins. 
\end{proof}
Now return to the theorem at hand. Fix a planar graph $H$ and $\varepsilon \in (0, 1)$. Let $n = |V(H)|$ and $m = \floor{1/\varepsilon + 1}^{n-1}$. Builder's strategy consists of two phases: making $m$ faces in the circle packing, then playing the box game on those $m$ faces. 

More precisely, in the first (at most) $m+2$ turns, Builder does the following: first place two circles tangent to each other, then for each of the remaining $m$ moves, choose some pair of tangent circles $\Omega_1$ and $\Omega_2$, and place a circle $\Omega$ tangent to both $\Omega_1$ and $\Omega_2$. Let $\mathcal C$ denote the circle packing at this point; evidently $\mathbf G(\mathcal C)$ has at least $m$ faces $F_1, \ldots, F_m$. 

By the Koebe-Andreev-Thurston circle packing theorem (see \cite{PA95} for a proof), there is a circle packing $\mathcal C_H$ such that $\mathbf G(\mathcal C_H) \cong H$. For each face $F_i$, let $\mathcal C_i = \{\Omega_{i,j}\}_{j=1}^n$ be a set of circles such that $\Omega_{i,j}$ lies inside $F_i$ for all $1 \leq j \leq n$, and $\mathbf G(\mathcal C_i) \cong H$. Such $\mathcal C_i$ exist by the circle packing theorem. Builder wins if he manages to place all $n$ of the circles in $\mathcal C_i$ for some $i$ before Spoiler has placed a single circle in the face $F_i$. But this is precisely the $(1+\varepsilon) : 1$ biased $(m, n)$-box game, and by \Cref{lem:boxgame}, Builder has a winning strategy. 

\section{Proof of \Cref{thm:diameter}}
\begin{figure}[t]
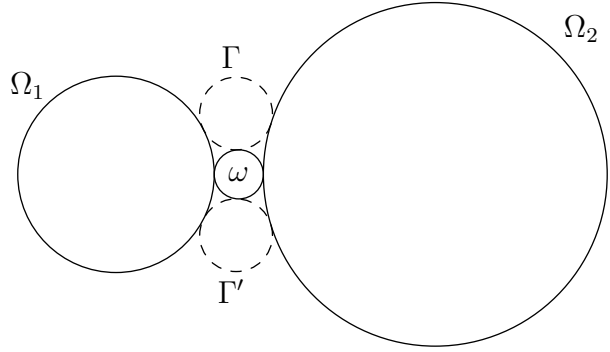

\begin{asy}
size(8cm);
pair A = (-5, 0), B = (8, 0), C = (0, 0), D = (-1/9, 8 * sqrt(70) / 27), E = conj(D);
real rA = 4, rB = 7, rC = 1, rD = length(D - A) - rA, rE = rD;
draw(circle(A, rA));
draw(circle(B, rB));
draw(circle(C, rC));
draw(circle(D, rD), dashed);
draw(circle(E, rE), dashed);
label("$\Omega_1$", A + 0.9 * (-rA, rA));
label("$\Omega_2$", B + 0.85 * (rB, rB));
label("$\omega$", C);
label("$\Gamma$", D + (-0.2, rD + 0.8));
label("$\Gamma'$", E + (-0.2, - rD - 0.8));
\end{asy}
\centering
\caption{Two threats, $\Gamma$ and $\Gamma'$, at once}
\label{fig:winning}
\end{figure}
\begin{definition}
Let $\mathcal C$ be a circle packing. We say that $\mathcal C$ is a \textit{winning position} (as shown in \Cref{fig:winning}) if there are circles $\Omega_1, \Omega_2, \omega \in \mathcal C$ such that:
\begin{itemize}
    \item $\omega$ is tangent to $\Omega_1$ and to $\Omega_2$;
    \item the two circles $\Gamma$ and $\Gamma'$ tangent to all three circles $\Omega_1$, $\Omega_2$, and $\omega$ are such that $\mathcal C \cup \{\Gamma, \Gamma'\}$ is a circle packing (recall that this means $\Gamma,\Gamma'$ do not intersect the interiors of any circles in $\mathcal C$);
    \item none of the new inner faces created by adding $\Gamma$ or $\Gamma'$ to $\mathcal C$ contain any circles in $\mathcal C$; and
    \item there is no circle $\gamma$ such that $\mathcal C\cup \{\gamma\}$ is a circle packing and $\gamma$ intersects the interiors of both $\Gamma$ and $\Gamma'$. 
\end{itemize}
For clarity, we may also say that $\mathcal C$ is a \textit{winning position with witness} $(\Omega_1, \Omega_2, \omega)$. 
\end{definition}

\begin{lemma}
If Builder can reach a winning position, then he can construct arbitrarily large Apollonian networks. 
\end{lemma}
\begin{proof}
Suppose that immediately after his $i$th turn, Builder achieves a winning position $\mathcal C_i$ with witness $(\Omega_1, \Omega_2, \omega)$. Let $\Gamma, \Gamma'$ be the two circles tangent to all three of $\Omega_1$, $\Omega_1$, and $\omega$. On his $(i+1)$th turn, Builder is threatening to play either $\Gamma$ or $\Gamma'$. At the current turn, these are legal moves since, by definition of a winning position, $\mathcal C_i \cup \{\Gamma, \Gamma'\}$ is a circle packing. Moreover, Spoiler can block at most one of these threats since, again by assumption, Spoiler cannot legally place a circle $\gamma$ which intersects the interior of both $\Gamma$ and $\Gamma'$. 

It follows that, once Builder plays his $(i+1)$th move, the new circle packing $\mathcal C_{i+1}$ contains two distinct $K_3$'s. By definition, neither of the two faces bounded by these $K_3$'s contain any circles in $\mathcal C_{i+1}$. Therefore, both $K_3$'s threaten to become $K_4$'s on Builder's next turn, and Spoiler cannot stop them both since these threats lie in different faces. 

Thus, after Builder's $(i+2)$th move, the circle packing $\mathcal C_{i+2}$ now contains a $K_4$. Moreover, observe that no three of the inner faces of this $K_4$ contain any circles in $\mathcal C_{i+2}$. It follows that, for all of Builder's future moves, Builder can play a circle in an unoccupied face of the current Apollonian network, thereby repeatedly increasing its size. 
\end{proof}

It remains to show that Builder can reach a winning position. Let Spoiler move first (the argument is unaffected if instead Builder moves first), and without loss of generality suppose Spoiler plays a unit circle $\Omega_1$, centered at $(0, -1)$ and passing through $(0,0)$. The idea is that Builder will place a circle very close to $\Omega_1$ so that it will impossible for Spoiler to prevent Builder from achieving a winning position on the very next turn. 

The implementation details of this argument follow the flavor of \Cref{lem:bigcircle}. Fix $\delta > 0$ such that the parabola $x \mapsto -\delta x^2$ lies below the upper half of $\Omega_1$ over some neighborhood $(-a, a)$ of 0. We now fix some constants which will be used for quantitative estimates later on. Pick some $0 < x_0 < a$ and $\varepsilon > 0$ such that
\[
14\delta x_0 < \frac12 \quad \text{and} \quad \frac{7\varepsilon}{x_0} < \frac12.
\]
Now, let $\Omega_2$ be the circle centered at $(0, 1+\varepsilon)$ with radius 1, and note that the parabola $x \mapsto \varepsilon + \delta x^2$ lies above the lower half of $\Omega_2$ over the interval $(-a, a)$. Then:
\begin{center}
    Builder plays the circle $\Omega_2$. 
\end{center}
We claim that Builder will achieve a winning position on the very next turn. Let $\ell$ and $\ell'$ denote the vertical lines $x = -x_0$ and $x = x_0$ respectively.

Let $\omega_0$ be tangent to $\ell, \Omega_1, \Omega_2$; let $\omega_1$ be tangent to $\omega_1, \Omega_1, \Omega_2$; and let $\omega_2$ be tangent to $\omega_1, \Omega_1, \Omega_2$ in the manner shown in \Cref{fig:gap}. For $i=0,1,2$, let $\omega_i'$ be the reflection of $\omega_i$ in the $y$-axis. 

\begin{figure}
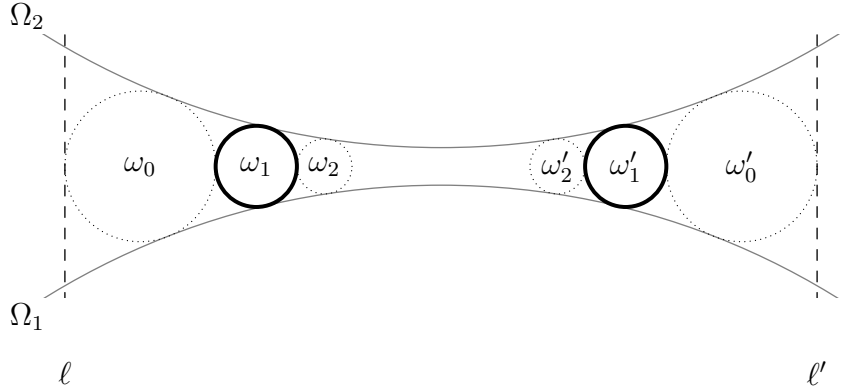

\begin{asy}
size(12cm);
real eps = 0.05, bX = 0.6, bY = 0.2, a = 0.5;
path border = (bX, bY)--(bX, -bY+eps)--(-bX, -bY+eps)--(-bX, bY)--cycle;
pair A = (0, 1 + eps);
draw(circle(A, 1)^^circle((0, -1), 1), gray);
draw((a, -bY-2*eps)--(a, bY)^^(-a, -bY-2*eps)--(-a, bY), dashed);

pair O0 = (-0.39979, eps/2);
real r0 = length(A - O0) - 1;
pair O1 = (-0.24558, eps/2);
real r1 = length(A - O1) - 1;
pair O2 = (-0.15493, eps/2);
real r2 = length(O1 - O2) - r1;
draw(circle(O0, r0), dotted);
draw(circle(O1, r1), linewidth(1.5));
draw(circle(O2, r2), dotted);
draw(circle(-conj(O0), r0), dotted);
draw(circle(-conj(O1), r1), linewidth(1.5));
draw(circle(-conj(O2), r2), dotted);
label("$\omega_1$", O1);
label("$\omega_1'$", -conj(O1));
label("$\omega_2$", O2);
label("$\omega_2'$", -conj(O2));
label("$\omega_0$", O0);
label("$\omega_0'$", -conj(O0));

clip(currentpicture, border);
label("$\Omega_1$", (-a - eps, -bY+eps/2));
label("$\Omega_2$", (-a - eps, bY+eps/2));
label("$\ell$", (-a, -bY-eps));
label("$\ell'$", (a, -bY-eps));
\end{asy}
\centering
\caption{Forcing a winning position}
\label{fig:gap}
\end{figure}

The proof of correctness now lies in the following claim. 

\begin{claim}\label{claim:gap}
There is no circle $\gamma$ satisfying:
\begin{itemize}
    \item $\gamma$ intersects the interior of both $\omega_i$ and $\omega_j'$ for some $i,j$; and
    \item $\{\Omega_1, \Omega_2, \gamma\}$ is a circle packing. 
\end{itemize}
\end{claim}
\begin{proof}
Suppose there were such a $\gamma$. Then $\gamma$ must lie in the face bounded by $\Omega_1, \Omega_2, \ell, \ell'$. Note that the radius of each of the circles $\{\omega_k, \omega_k'\}_{k=0}^2$ and $\gamma$ is at most the vertical distance between $\Omega_1$ and $\Omega_2$ at $x=x_0$, which in turn is at most the vertical distance between the parabolas $x \mapsto -\delta x^2$ and $x \mapsto \varepsilon + \delta x^2$ at $x=x_0$, namely $\varepsilon + 2\delta x_0^2$. Then the sum of the diameters of all 7 circles is at most $14(\varepsilon + 2\delta x_0^2)$. However, the distance between $\ell$ and $\ell'$ is $2x_0$, and by our choice of $x_0$ and $\varepsilon$, we have
\[
\frac{14(\varepsilon + 2\delta x_0^2)}{2x_0} = \frac{7\varepsilon}{x_0} + 14\delta x_0 < \frac12 + \frac12 = 1,
\]
impossible. 
\end{proof}
The upshot of \Cref{claim:gap} is that no matter how Spoiler moves, Builder will achieve a winning position on his next move. Suppose Spoiler plays a circle $\gamma$. Then by \Cref{claim:gap}, $\gamma$ is either internally disjoint with each of $\{\omega_0, \omega_1, \omega_2\}$ or internally disjoint with each of $\{\omega_0', \omega_1', \omega_2'\}$. Without loss of generality, assume the former. Then Builder plays $\omega_1$, and clearly $\{\Omega_1, \Omega_2, \omega_1, \gamma\}$ is a winning position with witness $(\Omega_1, \Omega_2, \omega_1)$.

\section{Further Questions}
We conclude by asking the following unresolved questions:
\begin{enumerate}
\item Can we close the gap between \Cref{thm:ham21} and \Cref{thm:ham13} and determine the threshold for the Hamiltonian cycle game?
\item \Cref{thm:diameter} exhibits graphs which are constructible in the circle packing game but not the edge drawing game. Are there graphs constructible in the edge drawing game but not the circle packing game?
\item It is easy to see that Builder can construct balanced binary trees of arbitrary depth in the circle packing game. Is this possible in the edge drawing game? Similarly, \Cref{thm:apollonian} shows that Builder cannot construct the pentagonal prism in the edge drawing game. Can Builder do so in the circle packing game? 
\end{enumerate}

\printbibliography

\end{document}